\documentclass[10pt,leqno]{amsart}

\usepackage{amssymb}
\usepackage{mathrsfs}
\usepackage{color}
\usepackage{soul}

\def\xx{\langle x \rangle}

\newcommand\N{{\mathbb N}}
\newcommand\R{{\mathbb R}}

\newcommand\C{{\mathbb C}}

\def\AA{{\mathcal A}}
\def\BB{{\mathcal B}}
\def\CC{{\mathcal C}}
\def\DD{{\mathcal D}}
\def\EE{{\mathcal E}}

\def\LL{{\mathcal L}}

\def\OO{{\mathcal O}}

\def\RR{{\mathcal R}}

\def\TT{{\mathcal T}}

\def\VV{{\mathcal V}}

\def\M{{\mathbb M}}

\def\BBB{{\mathscr{B}}}

\def\eps{{\varepsilon}}

\let\phi=\varphi
\let\epsilon=\varepsilon
\let\tilde=\widetilde
\let\Bbb=\mathbb

\newcommand{\wto}{\rightharpoonup}

\newtheorem{theo}{Theorem}
\newtheorem{prop}[theo]{Proposition}
\newtheorem{lem}[theo]{Lemma}
\newtheorem{cor}[theo]{Corollary}

\theoremstyle{definition}
\newtheorem{rem}[theo]{Remark}

\newcommand{\beqn}{\begin{equation}}
\newcommand{\eeqn}{\end{equation}}
\newcommand{\bear}{\begin{eqnarray}}
\newcommand{\eear}{\end{eqnarray}}
\newcommand{\bean}{\begin{eqnarray*}}
\newcommand{\eean}{\end{eqnarray*}}

\def\Nt{|\hskip-0.04cm|\hskip-0.04cm|}

\begin{document}

\title[Subcritical Fokker-Planck equation]{The Fokker-Planck equation \\
with subcritical confinement force}

\author{O. Kavian, S. Mischler, M. Ndao} 

\address[Otared Kavian]{Universit\'e Paris--Saclay; UVSQ \& CNRS, UMR 8100; Laboratoire de Math\'e\-matiques de Versailles; 45 avenue des Etats Unis; 78035 Versailles cedex, France.}
\email{kavian@math.uvsq.fr}

\address[St\'ephane Mischler]{Universit\'e Paris-Dauphine, Institut Universitaire de France (IUF), PSL Research University,
CNRS, UMR [7534], CEREMADE, 
Place du Mar\'echal de Lattre de Tassigny
75775 Paris Cedex 16, 
France.}
\email{mischler@ceremade.dauphine.fr}

\address[Mamadou Ndao]{Universit\'e Paris--Saclay; UVSQ \& CNRS, UMR 8100; Laboratoire de Math\'e\-matiques de Versailles; 45 avenue des Etats Unis; 78035 Versailles cedex, France.}
\email{ndao75019@yahoo.fr}

\begin{abstract} We consider the Fokker-Planck equation 
with subcritical confinement force field which may not derive from a potential function. 
We prove the existence of an equilibrium (in the case of a general force) and
we establish some subgeometric 
rate of convergence to the equilibrium 
(depending on the space to which belongs the initial datum) in many spaces. 
Our results generalize similar results introduced by Toscani, Villani \cite{MR1751701} and 
R\"ockner, Wang \cite{MR1856277} for some forces associated to a potential and 
extended by Douc, Fort, Guillin \cite{MR2499863} and Bakry,  Cattiaux, Guillin \cite{MR2381160}
for some general forces:  the spaces are more general, 
the rates are sharper.

\end{abstract}

\maketitle

\bigskip
\begin{center}
{\bf Version of \today}
\end{center}

\bigskip

One-parameter semigroups and linear evolution equations [See also
34G10, 34K30], 35P15 Estimation of eigenvalues, upper and lower
bounds, 47H20 Semigroups of nonlinear operators [See also 37L05,
47J35, 54H15, 58D07], 35Q84 Fokker-Planck equations.


%
%
%
%
%
	
\vspace{0.3cm}



\bigskip


\tableofcontents


%
%
%

\section{Introduction} 
\label{sec:intro}
\setcounter{equation}{0}
\setcounter{theo}{0}

In the present work,  we consider the Fokker-Planck equation 
\beqn\label{eq:FPeq} 
\partial_t f =  \LL f = \Delta f + {\rm div} (f\,{\bf F})
\eeqn
on the density function $f = f(t,x)$, $ t > 0$, $x \in \R^d$, $d \ge 1$, in the case of a subcritical confinement force ${\bf F}$, and which is complemented with an initial condition
\beqn\label{eq:CondInit} 
f(0,x) = f_0 (x), \quad \forall \, x \in \R^d.
\eeqn
More precisely, we will always assume that the force field ${\bf F} \in C^1(\R^d,\R^d)$ satisfies  
\beqn\label{eq:CondChamp} 
x \cdot {\bf F}(x) \ge { |x| \, \langle x \rangle^{\gamma-1}}, \quad {\rm div} ({\bf F}(x))  \le C_F \, |x|^{\gamma-2}, \quad \forall \, x \in B_{R_0}^c\, ,
\eeqn
as well as
 \beqn\label{eq:CondChamp2} 
 |D {\bf F}(x)|  \le C'_F \,\langle x \rangle^{\gamma-2}, \quad \forall \, x \in \R^d, 
\eeqn
for some constants $C_F \geq d$, $R_0 > 0$, $C'_F > 0$ and an exponent  
\begin{equation}\label{eq:CondChamp2-b}
 \gamma \in (0,1).
\end{equation}
Here and below,  we denote $\langle x \rangle := (1 + |x|^2)^{1/2}$ for any $x \in \R^d$. 

\smallskip
It is worth mentioning that we have made two normalization 
hypotheses by taking a diffusion coefficient equal to $1$ in \eqref{eq:FPeq} as well as a lower bound constant equal to $1$ in the first condition in  \eqref{eq:CondChamp}. Of course, these two normalization hypotheses can be removed by standard scaling arguments (in the time and position variables) 
and thus do not restrict the generality of our analysis, but on the other hand noticeably simplify the presentation. 

\smallskip

 A typical example of a force field  is the one associated to a confinement potential 
\beqn\label{eq:def:F=nablaV}
{\bf F}(x) := \nabla V(x), \quad  
V(x)  := {  \langle x \rangle^\gamma \over \gamma} + V_0, \quad V_0 \in \R. 
\eeqn
In this case, we may observe that 
\beqn\label{eq:def:G=e-V}
G (x) := {\rm e}^{- V(x)  }  \in L^1(\R^d) \cap C^2(\R^d), 
\eeqn
is a stationary solution of \eqref{eq:FPeq}, and even an equilibrium state. We may  assume that $G$ is a probability measure, by choosing the constant $V_0$ adequately. We recall that when $\bf F$ is given by \eqref{eq:def:F=nablaV} with $\gamma \ge 1$, the following Poincar\'e inequality  
$$
\exists \, c > 0, \qquad \int_{{\Bbb R}^d}|f(x)|^2 \exp(-V(x))dx \leq c\, \int_{{\Bbb R}^d}|\nabla f(x)|^2\exp(-V(x)), 
$$
holds for any $f$ such that $\int_{{\Bbb R}^d}f(x)\exp(-V(x))dx = 0$. Such a Poincar\'e inequality  does not hold when $\gamma \in (0,1)$, which is the case studied in this paper,  but only a weak version of this inequality  remains true (see \cite{MR1856277}, and below \eqref{eq:wPineq} and section~\ref{subsec:RocknerWang}).
In particular, there is no spectral gap for the associated operator $\LL$, nor is there an exponential trend to the equilibrium for the associated semigroup. Similarly,  the classical {\it logarithmic Sobolev inequality} does not hold but only a {\it modified} version of it, see the discussion in   \cite[Section 2]{MR1751701}.

\smallskip
In the general case of a  force field which is not the gradient of a potential, one may see easily that the above Fokker-Planck equation preserves positivity, that is 
$$
 f(t,.) \ge 0, \quad \forall \, t \ge 0, \quad \mbox{if}\quad f_0 \ge 0, 
$$
and that it conserves mass, that is 
\beqn\label{eq:mass}
\M(f(t,\cdot)) = \M(f_0),  \quad\forall \, t \ge0, \quad\mbox{with}\quad \M(g) := \int_{\R^d} g(x) \, dx.
\eeqn
Moreover, the Fokker-Planck operator $\LL$ generates a (Markov) semigroup in many Lebesgue spaces which has a unique positive normalized steady state (or invariant probability measure). 

\smallskip
Before stating our existence result, let us introduce some notation. 
For any exponent $p \in [1,\infty]$, we define the polynomial and 
exponential weight functions $m : \R^d \to \R_+$,  by 
\beqn\label{eq:def:weightpoly}
m(x) := \langle x \rangle ^k , \qquad\mbox{for some }\, k > \ k^* ,\;\mbox{with }\, k^* :=  \max(d,C_{F})/p', 
\eeqn
where $p' := p/(p-1)$ is the conjugated exponent associated to $p$, 
\beqn\label{eq:def:weightexpo1}
m(x) := \exp(\kappa \, \langle x \rangle^s),\qquad \mbox{for some }\, 0 < s < \gamma\, \mbox{ and }\,\kappa > 0,
\eeqn
or 
\beqn\label{eq:def:weightexpo2}
m(x) := \exp(\kappa \, \langle x \rangle^\gamma), \qquad\mbox{for some }\, \kappa \in (0,1/\gamma), 
\eeqn
as well as the associated Lebesgue spaces
$$
L^p(m) = \{  f \in L^1_{\rm loc}({\Bbb R}^d); \,\, \|  f \|_{L^p(m)}  := \|  f m \|_{L^p} < \infty \}. 
$$
We also use the shorthands $L^p_k = L^p(m)$ when $m(x) = \langle x \rangle ^k$. 
It is noteworthy that, for such a choice of weights $m$, we  have $L^p(m) \subset L^1({\Bbb R}^d)$.

\smallskip
As a first step, we have the following existence and uniqueness result. 

\begin{theo}\label{th:wellposedness} For any exponent $p \in [1,\infty]$, any weight function $m$ satisfying either of definitions \eqref{eq:def:weightpoly},   \eqref{eq:def:weightexpo1}
or  \eqref{eq:def:weightexpo2}, and any initial datum $f_0 \in L^p(m)$, there exists a unique global solution $f$ to  the  Fokker-Planck equation \eqref{eq:FPeq}--\eqref{eq:CondInit}, such that for any $T > 0$,
$$
f\in C([0,T];L^1({\Bbb R}^d)) \cap L^\infty(0,T;L^p(m)).
$$
Moreover the associated flow preserves positivity and  conserves mass. 
Also, the operator $\LL$ generates a strongly continuous semigroup $S_\LL(t)$ in $L^p(m)$ when $p \in [1,\infty)$. 

On the other hand, there exists a unique positive, unit mass, stationary solution $G$ such that  
$$
G \in L^1_1, \quad \M(G) = 1
\qquad\mbox{and}\qquad
\Delta G + {\rm div}(G\, {\bf F}) = 0.
$$
More precisely, there exists $\kappa^*_1 \ge \kappa^*_0 := 1/\gamma$ such that  for any $\kappa_0 \in (0,\kappa^*_0)$, $\kappa_1 \in (\kappa^*_1,\infty)$, there holds
\beqn\label{eq:EstimGLinfty}
C_1 {\rm e}^{-\kappa_1  \langle x \rangle^\gamma} \le G \le C_0 {\rm e}^{-\kappa_0  \langle x \rangle^\gamma} \quad\mbox{on}\quad \R^d,
\eeqn
for some constructive constants $C_0,C_1 > 0$.
\end{theo}

The well-posedness for the evolution equation is not surprising and the proof follows classical dissipativity arguments. 
Due to the lack of compactness of the associated semigroup, the standard Krein-Rutman theory does not apply directly in the case $\gamma \in (0,1)$, and the existence of a stationary solution is not straightforward. We  follow here a similar fixed point theorem strategy as in the recent work \cite{NM*} (see also \cite{GPV,EMRR} and the references therein) where the Fokker-Planck equation with general force field \eqref{eq:FPeq}--\eqref{eq:CondChamp2-b} in the case $\gamma \ge 1$ is considered.  
Our approach provides an alternative deterministic proof of the existence of a stationary solution which may also be established using the probabilistic approach developed by Douc, Fort and Guillin in \cite{MR2499863}. Another  deterministic approach is also presented in \cite{CanizoMischler} (which applies to much more general Markov semigroups). 

\smallskip
 Once the existence of a stationary solution is known, we are interested in the long time behaviour of the solution $f(t,\cdot)$ to  the Fokker-Planck equation \eqref{eq:FPeq}.  For the sake of clarity, we consider separately the following two cases. 
 
\smallskip
\noindent{\bf Case 1. }  Following \cite[Example 1.4(c)]{MR1856277}, we consider the case when
furthermore  the above steady state $G$ fulfills a {\it weak weighted Poincaré-Wirtinger inequality}. More precisely, in this case we assume that there exist some constants $R_0, c_1, c_2 > 0$ such that the function $ V := - \log G \in C^1(\R^d)$ satisfies 
\beqn\label{eq:Vsimxgamma}
\forall \, x \in B^c_{R_0}, \qquad c_1 \, |x|^\gamma \le V(x) \le c_2 \, |x|^\gamma, 
\eeqn
and also that there exists $\mu > 0$ such that for any $f \in \DD(\R^d)$ with $\M(f) = 0$, the following {\it weak weighted Poincaré-Wirtinger inequality}
holds
\beqn\label{eq:wPineq}
\int |\nabla (f/G)|^2 \, G \, dx \ge \mu \, \int f^2 \, \langle x \rangle^{2\gamma-2} \, G^{-1} \, dx.
\eeqn
Inequality \eqref{eq:wPineq} holds when $c_1 = c_2$ in \eqref{eq:Vsimxgamma}, see Lemma~\ref{lem:wwPineq:case1}. Although this inequality probably belongs to folklore, we were not be able to find a precise reference. We however refer to  \cite[Theorem~3.3]{MR1856277} and \cite[Theorem~2.18]{MR2609591} where related inequalities are established. 

The weak Poincar\'e inequality \eqref{eq:wPineq} is a consequence of a {\it ``local Poincar\'e inequality"} (or {\it ``Poincaré-Wirtinger inequality"}) together with the fact that  the following Lyapunov condition (see for instance \cite{MR2381160,MR2386063})
\beqn\label{eq:LyapCond}
\Delta w - \nabla V \cdot \nabla w \le  (- \zeta(x) + M \chi_R)w,  \quad \forall \, x \in \R^d,
\eeqn
holds for some well chosen function $w : \R^d \to [1,\infty)$. Here it is assumed that $M$ and $R$ are two positive constants, $\chi_R(x) := \chi(x/R)$ is a truncation function defined through a certain $\chi \in \DD(\R^d)$ such that ${\bf 1}_{[|x|\le 1]}  \le \chi \le {\bf 1}_{[|x| \le 2]}$ and  $\zeta (x) : = \zeta_0  \, \langle x \rangle^{2(1-\gamma)}$ for a constant $\zeta_{0} >0$. 
It is worth emphasizing that, in this case, the force field ${\bf F}$ can be written as 
\beqn\label{eq:condSym1}
{\bf F} = \nabla V + {\bf F}_0, \qquad {\rm div}( {\rm e}^{-V} \, {\bf F}_0) = 0, 
\eeqn
with no other specific condition on ${\bf F}_0$ except that  ${\bf F}$ still satisfies conditions \eqref{eq:CondChamp}--\eqref{eq:CondChamp2-b}. Under these circumstances, we can give 
a simpler proof than in the general case. 
\smallskip

\noindent {\bf Case 2. }  This corresponds to the general case when ${\bf F}$ satisfies only conditions \eqref{eq:CondChamp}--\eqref{eq:CondChamp2-b}, 
without any further assumption on  the stationary state $G$, which in general cannot be determined explicitly. Using the above notations for $M,R$, $\chi_{R}$ and $\zeta$ introduced in the inequality \eqref{eq:LyapCond}, the assumptions \eqref{eq:CondChamp}--\eqref{eq:CondChamp2-b} made on ${\bf F}$ imply in particular the following inequality 
\beqn\label{eq:CondChamp3}
\LL^* m^p := \Delta m^p - {\bf F} \cdot\nabla m^p  \le (-\zeta(x)+ M \chi_R) m^p , \quad \forall \, x \in \R^d,
\eeqn
which is  another version of the  Lyapunov condition \eqref{eq:LyapCond}.  

\smallskip

The main and fundamental difference between these two cases is that the first one involves assumptions on the equilibrium state $G = {\rm e}^{-V}$  while the second one only involves an assomption on the force field ${\bf F}$.

\medskip

In the sequel, when $a(t) \geq 0$ and $b(t) \geq 0$ are two functions of time $t > 0$, we write $a(t) \lesssim b(t)$ to mean that there exists a positive constant $c_{0}$ independent of $t$ such that one has $a(t) \leq c_0\, b(t)$ for all $t > 0$.

\medskip
Our main result is as follows:  

\begin{theo}\label{th:MAIN} 
Let ${\bf F}$ satisfy \eqref{eq:CondChamp}--\eqref{eq:CondChamp2-b},
and let $p \in [1,\infty]$. For a weight function $m$ satisfying either of definitions \eqref{eq:def:weightpoly},   \eqref{eq:def:weightexpo1} or  \eqref{eq:def:weightexpo2}, we define the subgeometric rate function $\Theta_{m}$ as follows: when $m(x) = \xx^k$, we take $\beta \in (0, (k-k^*)/(2-\gamma))$ arbitrary and we set
\beqn\label{eq:MainEstimPoly}
\Theta_{m}(t) := (1 + t)^{-\beta}.
\eeqn
 When $m(x) = \exp(\kappa \, \langle x \rangle^s)$,  we take $ \sigma  := s/(2-\gamma)$ and we set for a $\lambda > 0$
\beqn\label{eq:MainEstimExpo1}
\Theta_{m}(t) := \exp(- \lambda t^\sigma).
\eeqn
Then for any initial datum $f_0 \in L^p(m)$, the associated solution $f = f(t,x)$  to the Fokker-Planck equation \eqref{eq:FPeq} satisfies
\beqn\label{eq:MainEstim}
\|f(t,.) - \M(f_0)\, G\|_{L^p } \lesssim \Theta_m(t) \, \|f_0 - \M(f_0)\, G \|_{L^p(m)},
\eeqn
where we recall that $\M(f_0)$ denotes the mass of $f_0$ defined in \eqref{eq:mass}.
\end{theo}

\begin{rem}\label{Rk:Main1qua} We believe that Theorem~\ref{th:MAIN} is new in the sense that the decay estimate \eqref{eq:MainEstimPoly}-\eqref{eq:MainEstim} is optimal and constructive and that  the Lebesgue spaces involded appear with the same exponent $p \in [1,\infty]$. This last fact is of main importance when one is interested in the stability of nonlinear evolution PDEs, see for instance \cite{KleberSM}.  
\end{rem}

We discuss below in details  some of the existent literature about the same kind of results, some possible extensions and the general strategy of our approach. 

\begin{rem}\label{Rk:Main1} When the force field ${\bf F}$ and the equilibrium $G$ satisfy similar assumptions as  described in {\bf Case 1},  several previous results are known. 
A less accurate  rate of decay than the one given by \eqref{eq:MainEstimPoly}-\eqref{eq:MainEstim}, actually a decay rate of order $\OO(t^{-(k-2)/(2(2-\gamma))})$ in $L^1$-norm, has been proved  by G. Toscani, C. Villani in \cite[Theorem 3]{MR1751701} under the additional assumptions that the initial datum $f_0$ is nonnegative, has finite energy and finite Boltzmann entropy. 
The proof of this first result relies on a modified and  weak  version of the log-Sobolev inequality. 
(It is known that the standard (and stronger) version of the log-Sobolev inequality  only  holds when $\gamma \ge 2$).
In \cite[Example 1.4 (c) and Section 3 for the proof]{MR1856277}, M.~R\"ockner, F.Y. Wang have established the exponential rate 
\beqn\label{eq:RWestim}
 \|f (t,\cdot) - \M(f_{0})G\|_{L^2} \lesssim  
{\rm e}^{-\lambda t^{\sigma^{\sharp}}} \| f_{0} - \M(f_{0})G \|_{L^\infty(G^{-1})}, 
\eeqn
with  $\sigma^{\sharp} := \gamma/(4-3\gamma)$, some $\lambda > 0$ and any $f_0 \in L^\infty(G^{-1})$.  The proof of \eqref{eq:RWestim} is based on a  weak Poincar\'e inequality and the existence of a family of entropy functionals. 

Estimate \eqref{eq:MainEstimPoly}-\eqref{eq:MainEstim} is more accurate than \eqref{eq:RWestim}  since, when the weight function $m(x) := \exp(\kappa\xx^\gamma)$ we have $\sigma = \gamma/(2-\gamma)$, and the function  $\Theta_{m}(t) :=  \exp(- \lambda t^{\gamma/(2-\gamma)})$ decays faster than $\exp(-\lambda t^{\sigma^{\sharp}})$. Also, more importantly, the same  Lebesgue exponent $p$ on both  sides of the inequality is involved in \eqref{eq:MainEstim} while it is not the case in \eqref{eq:RWestim}. 
Also, as we shall see below, from \eqref{eq:MainEstimPoly}-\eqref{eq:MainEstim} one easily deduces by interpolation that for any $\theta \in (0,1)$ one has 
$$\|f(t,.) - \M(f_0)\, G\|_{L^p(m^\theta) } \lesssim \Theta_m(t) \, \|f_0 - \M(f_0)\, G \|_{L^p(m)}.$$

\end{rem}

\begin{rem}\label{Rk:Main1bis} D. Bakry, P. Cattiaux \& A. Guillin in \cite{MR2381160} have extended the approach of M. R\"ockner, F.Y. Wang to the case of a general force field. 
More precisely,  they establish an inequality, which they call {\it weak Lyapunov-Poincar\'e inequality}, in \cite[Theorem~3.10 and Section~4.2.1]{MR2381160} from which they deduce the same rate of convergence \eqref{eq:RWestim}. 
 That work is based on the Lyapunov condition method,  which we have already discussed in the presentation of {\bf Case 1} and which is closely related to the method we use here. The Lyapunov condition method  has then been extensively studied during the last decade, and we refer for instance to \cite{MR2381160,MR2386063,MR2609591}  and the references therein for more details. 
\end{rem}

\begin{rem}\label{Rk:Main1ter} A completely different approach has been developed by R. Douc, G. Fort \& A. Guillin in \cite{MR2499863} generalizing the classical Harris-Meyn-Tweedie theory to the weak confinement setting and  adapted to both the potential case and the general force case. We also refer to the notes on the convergence of Markov processes by M. Hairer, and more precisely to \cite[Theorem~4.1]{Hairer}, for a more explicite formulation of the subgeometric convergence result and a simplified proof (starting from a slightly different set of hypotheses). 
It provides estimate \eqref{eq:MainEstim} in the case $p:=1$ and $m := G^{-1}$ (with same rate), see \cite[(7.3)]{Hairer}. The drawback is that the approach is definitively probabilistic and the constants are maybe not so explicit. In a forthcoming paper \cite{CanizoMischler}, J.A. Ca\~nizo and S. Mischler provide a deterministic and elementary proof of the above mentioned theorem. 
We also refer to A. Eberle, A.~Guillin \& R. Zimmze \cite{EGZ} for recent related results. 
\end{rem}

\begin{rem} \label{Rk:Main2}
The same kind of decay estimates remains true when, on the left hand side of \eqref{eq:MainEstim}, the  Lebesgue norm $L^p$ is replaced by a weighted Lebesgue norm $L^{p}(m^\theta)$ with 
$0\leq \theta < 1$. More precisely,  when $m(x) = \xx^k$ and $\theta$ is such that $k^*/k < \theta < 1$, we choose $\beta := k(1-\theta)/(2-\gamma)$, and if $0 \leq \theta \leq k^*/k$, we choose $\beta \in (0, (k-k^*)/(2-\gamma))$ arbitrary. In both cases, we define $\Theta_m$ through \eqref{eq:MainEstimPoly}. 
When $m(x) = \exp(\kappa \, \langle x \rangle^s)$ the definition of the decay rate $\Theta_m$ is unchanged. 
\end{rem}

\begin{rem} \label{Rk:Main3}
When an exponential weight function $m(x) := \exp(\kappa \langle x \rangle^\gamma)$ is considered, one could have a field force ${\bf F}$ satisfying the first condition of \eqref{eq:CondChamp} and ${\rm div} ({\bf F}) \le C'_F \, \langle x \rangle^{\gamma' - 2}$, with $\gamma' < 2\gamma$,  or $\gamma' = 2\gamma$ but with $C'_F$ small enough, and obtain similar results. However we do not push our investigations in that direction, since the general ideas of the proof are essentially the same.   
\end{rem}

\begin{rem} \label{Rk:Main4} 
When a polynomial weight function $m(x) = \langle x \rangle^k$ is considered, the decay rate in \eqref{eq:MainEstim} is given by \eqref{eq:MainEstimPoly}, which is  better than the decay rate one might obtain by a mere interpolation argument between $L^2(\exp(\kappa \langle x \rangle^\gamma))$ and $L^1({\Bbb R}^d)$. More precisely, assume that when the weight function $m(x) = \exp(\kappa \langle x \rangle^\gamma)$, the function $\Theta_{m}$ being given by \eqref{eq:MainEstimExpo1} with $s:= \gamma/(2-\gamma)$, we have \eqref{eq:MainEstim}, as well as the estimate $\|f(t)\|_{L^1} \lesssim \|f_{0}\|_{L^1}$ for any $f_0 \in L^1(\exp(\kappa\xx^\gamma))$ such that $\M(f_0) = 0$. Then for any $R > 0$ we have 
$\M(f_{0}1_{B_{R}}) = - \M(f_{0}1_{B_{R}^c})$, and thus we may write $f_{0} = f_{01} + f_{02} + f_{03}$ where
\begin{equation*}
f_{01} := \left(f_{0} - \M(f_{0}1_{B_{R}})\right)1_{B_{R}}, \quad
f_{02} := f_{0}1_{B_{R}^c}, \quad
f_{03} := \M(f_{0}1_{B_{R}^c})1_{B_{R}}.
\end{equation*}
Therefore for  $t > 0$, denoting by $f_{j}(t)$ the solution of \eqref{eq:FPeq} with initial datum $f_{0j}$, one has 
\begin{align*} 
\|f(t)\|_{L^1} & \leq  \|f_{1}(t)\|_{L^1} + \|f_{2}(t)\|_{L^1} + \|f_{3}(t)\|_{L^1} \\
&\lesssim \exp(-\lambda t^{\gamma/(2 -\gamma)})\, \|f_{01}\|_{L^1(\exp(\kappa\xx^\gamma))}+ \|f_{02}\|_{L^1} + \|f_{03}\|_{L^1} \\
& \lesssim \exp(-\lambda t^{\gamma/(2 -\gamma)}) \, 
\| (f_0 - \M(f_0 \, {\bf 1}_{B_R})) {\bf  1}_{B_R} \|_{L^1(\exp(\kappa\xx^\gamma))} 
\\
&\qquad\qquad
+ \|\M(f_0 \, {\bf 1}_{B_R^c}) {\bf 1}_{B_R}\|_{L^1} + \|f_0 \, {\bf 1}_{B_R^c} \|_{L^1} \\
&\lesssim \exp(-\lambda t^{\gamma/(2 -\gamma)})\, {\rm e}^{\kappa \, \langle R \rangle^\gamma} \,  \| f_0 \|_{L^1}
+ (R^{d-k} + R^{-k} ) \| f \|_{L^1_k} \\ 
&\lesssim \left(\exp(-\lambda t^{\gamma/(2 -\gamma)} + \kappa \langle R\rangle^\gamma) + R^{d-k}\right) \| f \|_{L^1_k}.
\end{align*}
Assuming that $t > (2\kappa/\lambda)^{(2-\gamma)/\gamma}$, we may choose $R$ so that $\kappa\langle R \rangle^\gamma =  \lambda t^{\gamma/(2-\gamma)}/2$, we find that when $k > d$, for any $t > (2\kappa/\lambda)^{(2-\gamma)/\gamma}$ we have
$$
\|f(t)\|_{L^1} \lesssim t^{{-(k-d)\over 2 - \gamma}} \, \|f_0 \|_{L^1_k}.
$$
This decay estimate is not as sharp as the one given by Theorem~\ref{th:MAIN} when the weight function is $m := \xx^k$, since according to the definition \eqref{eq:MainEstimPoly} in this case we have actually a decay rate of $\OO(t^{-K})$ for any $K \in (0,k/(2-\gamma))$. \qed
\end{rem}

\begin{rem} \label{Rk:Main5}
In a few previous papers, due in particular to R.E. Caflisch \cite{MR575897,MR576265}, T.M. Liggett \cite{MR1112402}, G. Toscani \& C. Villani~\cite{MR1751701}, Y. Guo~\cite{Guo} or K. Aoki \& F. Golse~\cite{AokiGolse}, a certain number of models, arising from statistical physics, has been considered for which only subgeometric (but not geometric) rate of decay to the equilibrium can be established. 
As it is the present case, one can associate to each of these models a  linear  operator which  does not enjoy any spectral gap in its spectrum set and that is the reason why exponential rate of convergence fails. 


An abstract theory for non-uniformly exponentially stable semigroups (with non exponential decay rate) has also been recently developed  and we refer the interested reader to \cite{BEPS,BattyD} and the references therein.
We  finally refer to K. Carrapatoso \& S. Mischler \cite{KleberSM} where similar semigroup analysis as here is developed and applied in order to establish the well-posedness of the Landau equation in large spaces. 
\end{rem}

Let us briefly explain the main ideas behind our method of proof. 
In {\bf Case 1}, and as a first step, we may use the argument introduced in M. R\"ockner \& F.Y. Wang  \cite{MR1856277} (see also O. Kavian \cite[Lemma 1.3]{OK-Kompaneets}) which we briefly recall now. 
We consider three  Banach spaces $E_2$, $E_1$ and $E_0$, such that $E_2 \subset E_1 \subset E_0 \subset L^1$, and more precisely $E_{1}$ is an interpolation space of order $1-1/\alpha$ between $E_{0}$ and $ E_{2}$ for some $\alpha \in (1,\infty)$, that is 
\begin{equation}\label{eq:Interpol-E}
\|f \|_{E_1} \le C_\alpha \, \| f \|_{E_0}^{1/\alpha} \, \| f \|_{E_2}^{1-1/\alpha}, \quad \forall \, f \in  E_{2},  
\end{equation}
and such that the semigroup $S_{\LL}(t)$ associated to the Fokker-Planck equation can be solved in each of these spaces. 
Moreover, assume that for any $f_0 \in  E_{2} $, the solution $S_\LL(t)f_{0}= f(t)$ to the Fokker-Planck equation \eqref{eq:FPeq}--\eqref{eq:CondInit} satisfies the following two differential inequalities
\begin{equation}\label{eq:Diff-Ineq}
\frac d {dt} \| f(t) \|_{E_1} \le - \lambda \, \|  f(t)  \|_{E_0}, \qquad \frac d {dt} \|  f(t)  \|_{E_2} \le 0,
\end{equation}
for some constant $\lambda > 0$. Using the fact that $\|f(t)\|_{E_{2}} \leq \|f_{0}\|_{E_{2}}$, as a consequence of the  above second differential inequality, together with \eqref{eq:Interpol-E}, we obtain the closed differential inequality
$$
\frac d {dt} \|f(t)\|_{E_1} \le - \lambda \, C_\alpha^{-\alpha} \, \| f_0 \|_{E_2}^{-(\alpha -1)}\|f(t)\|_{E_1}^{\alpha}.
$$
We may readily integrate this inequality and we obtain the estimate 
\beqn\label{eq:intro:estimfLE1E2}
\|f(t) \|_{E_1}  \lesssim t^{-1/(\alpha - 1)} \| f_0 \|_{E_2}. 
\eeqn
Now, choosing $E_1 = L^2(G^{-1/2})$, $E_0 := L^2(G^{-1/2} \langle x \rangle^{\gamma-1})$ and $E_2 = L^\infty(G^{-1})$, one may see that the first differential inequality in \eqref{eq:Diff-Ineq} is an immediate consequence of the weak Poincar\'e inequality \eqref{eq:wPineq}.
The second differential inequality is a kind of generalized relative entropy principle (see \cite{MR1856277,MR2162224}). 
The above estimate \eqref{eq:intro:estimfLE1E2} is a somewhat rough variant of estimate \eqref{eq:MainEstimExpo1}. It is noteworthy that for $\alpha \in (1,2)$, we get that the associated semigroup $S_\LL$ defined by $S_\LL(t) f_0 = f(t)$ satisfies in particular $\| S_\LL  \|_{E_2\to E_1} \in L^1(0,\infty)$.   

\smallskip 
We then generalize the decay estimate to a wider class of Banach spaces by adapting the  extension theory introduced  and developed in 
C. Mouho \cite{Mcmp},  M.P.~Gualdani, S. Mischler \& C. Mouhot \cite{GMM} and S. Mischler \& C. Mouhot \cite{MM*}, and used in M.~Ndao \cite{NM*} for the case $1 \leq \gamma \leq 2$.
In order to do so, we consider the two Banach spaces  $\EE_2 := L^p(m)$ and $\EE_1 := L^p$, and we introduce a splitting $\LL = \AA + \BB$, where $\AA$ is an appropriately defined bounded operator so that $\BB$ becomes a dissipative operator. Then, we show that for $i=1$ and $i=2$
$$\|S_\BB\AA\|_{\EE_i \to \EE_i} \in L^1(\R_+), \qquad
\|S_\BB  \|_{\EE_2 \to \EE_1} \in L^1(\R_+), \qquad
\| \AA S_\BB  \|_{\EE_1 \to \EE_1}  \in L^1(\R_+).$$
If ${\mathcal T}_{i}$, with $i=1,2$, are two given operator valued, measurable functions, defined on $(0,\infty)$, we denote by
$$\left({\mathcal T}_{1}*{\mathcal T}_{2}\right)(t) := 
\int_{0}^t {\mathcal T}_{1}(\tau){\mathcal T}_{2}(t-\tau)\,d\tau
$$
their convolution on ${\Bbb R}_{+}$. We then set $\TT^{(*0)} := I$,  ${\mathcal T}^{(*1)} := {\mathcal T}$ and,  for any $k \geq 2$,  ${\mathcal T}^{(*k)} := {\mathcal T}^{*(k-1)}*{\mathcal T}$.
We may show that for $n \in \N$ sufficiently large (actually $n \ge 1 + (d/2)$ is enough), we have 
\beqn\label{eq:intro:SBAnL1}
\| (S_\BB \AA)^{(*n)} \|_{E_1 \to \EE_1} \in L^1(\R_+),  \quad \| ( \AA S_\BB)^{(*n)}  \|_{\EE_1 \to E_1}  \in L^1(\R_+).
\eeqn
Then, from the usual Duhamel formula, the solution of  \eqref{eq:FPeq} can be written as 
$$
f(t) = S_{\BB}(t)f_{0} + \int_0^t S_{\BB}(t-\tau)\AA S_{\LL}(\tau)f_{0}\,d\tau.
$$
 Thus, using the above notations for the convolution of operator valued functions, we have $S_{\LL} = S_{\BB} + S_{\BB}*(\AA S_{\LL})$, and interchanging the role played by $\LL$ and $\BB$ in this expression, we get the following operator versions of Duhamel formulas
\begin{align}
S_{\LL} &= S_{\BB} + S_{\BB}*(\AA S_{\LL}) = S_{\BB} + (S_{\BB}\AA)*S_{\LL} \label{eq:Duhamel-1} \\
 & = S_{\BB} + S_{\LL}*(\AA S_{\BB}) = S_{\BB} + (S_{\LL}\AA)*S_{\BB}. \label{eq:Duhamel-2}
\end{align}
Upon replacing recursively $S_{\LL}$ in either of the expressions on the right hand side by either of the Duhamel's formula, we get, for instance:
\begin{align*}
S_\LL &= S_{\BB} + S_{\BB}*\AA 
\left\{ S_{\BB} + (S_{\BB}\AA)*S_{\LL}\right\} \\
& = 
S_{\BB} + (S_{\BB} \AA ) *S_{\BB} + (S_{\BB}\AA)^{(*2)}*S_{\LL}.
\end{align*}
By induction on the integers $n_{1} \ge 0$, $n_{2} \ge 0$ and $n_1 + n_2 \ge 1$, 
we thus obtain
\begin{equation} \label{eq:Duhamel-gene}
S_\LL = \sum_{k=0}^{n_{1}+n_{2}-1} S_\BB * (\AA S_\BB)^{(*k)}  + ( S_\BB \AA)^{(*n_{1})} *   S_\LL  *  (\AA S_\BB)^{(*n_{2})}. 
\end{equation}

Using the above formulas \eqref{eq:Duhamel-gene} and estimates \eqref{eq:intro:SBAnL1}, 
as well as the decay estimate \eqref{eq:intro:estimfLE1E2} for initial data in the space $E_2$, 
we conclude that $\| S_\LL  \|_{\EE_2\to \EE_1} \in L^1(\R_+)$ which is nothing but a rough version of the estimates presented in Theorem~\ref{th:MAIN}.
While the method leading to  \eqref{eq:intro:estimfLE1E2} in $E_i$ can be performed only in very specific (Hilbert) spaces, the above extension method is very general and 
can be used in a large class of Banach spaces $\EE_i$ (once we already know the decay in one pair of spaces $(E_1,E_2)$). 

\smallskip 
On the other hand, in {\bf Case 2}, we start proving an equivalent to estimate \eqref{eq:intro:estimfLE1E2} in one appropriate pair of (small) spaces. We then argue similarly as in the previous case. 
We present three different strategies in order to establish an equivalent to estimate \eqref{eq:intro:estimfLE1E2}. On the one hand, we may adapt the strategy of M. R\"ockner and F.Y. Wang introduced to deal with  {\bf Case 1} by using a generalization of the weak Poincar\'e inequality \eqref{eq:wPineq} which 
also holds in the case of a general force field and which has been established by D. Bakry, P. Cattiaux \& A. Guillin in \cite{MR2381160}. 

A second approach consists in using the generalization of the well-known Doeblin-Harris-Meyn-Tweedie theory by R. Douc, G. Fort \& A. Guillin in \cite{MR2499863} to the so-called {\it subgeometric} framework, namely to the case when exponential trend to the equilibrium is not available.

A third and more original way consists in adapting the Krein-Rutman theory to the present context. 
On the one hand, it is a simple version of the Krein-Rutman theory because the equation is mass conserving, a property which implies that the largest eigenvalue of $\LL$ is $\lambda_{1} = 0$.
On the other hand, it is not a classical version because the operator $\LL$ does not have a compact resolvent (however it has power-compact resolvent in the sense of J. Voigt \cite{Voigt80}) and, more importantly, $0$ is not  necessarily an isolated point in the spectrum. First adapting (from \cite{EMRR,MS*} for instance) some more or less standard arguments,
 we prove that  there exists $G$, a unique  stationary solution of \eqref{eq:FPeq} which is positive, has unit mass and is such that $G \in L^\infty(\exp(\kappa \xx^\gamma))$, for all $\kappa \in (0,1/\gamma)$. Next, we prove an estimate similar to \eqref{eq:intro:estimfLE1E2} by establishing a set of accurate estimates on the resolvent operators $\RR_\BB(z)$, $R_\LL(z)$ and by using the iterated Duhamel formula 
$$
S_\LL = \sum_{k=0}^{5} S_\BB * (\AA S_\BB)^{(*k)}  +  S_\LL  *  (\AA S_\BB)^{(*6)},
$$
together with the inverse Laplace formula
$$
S_\LL*(\AA S_\BB)^{(*6)} (t) = {{\rm i} \over 2\pi} \, {1 \over t^n} \int_{-{\rm i}\infty}^{+{\rm i}\infty} {\rm e}^{zt} {d^n \over dz^n} \bigl[ R_\LL(z) (\AA R_\BB(z))^6 \bigr] \, dz,
$$
which holds true for any time $t>0$ and any integer $n$.  
  
\medskip
To finish this introduction, let us describe the plan of the paper. 
In Section~\ref{sec:bornesB}, we introduce an appropriate splitting $\LL = \AA + \BB$ and present the main estimates on the semigroup $S_\BB$. 
In Section~\ref{sec:bornesL}, we deduce that the semigroup $S_\LL$ is bounded in the spaces $L^p(m)$. 
In Section~\ref{sec:Poincare}, the proof of Theorem~\ref{th:MAIN}  is carried out in the case when a weak Poincar\'e inequality \eqref{eq:wPineq} is satisfied ({\bf Case 1}).
In Section~\ref{sec:NonPoincareStatProblem}, the second part of the proof of Theorem~\ref{th:wellposedness} on the stationary problem is presented  in the general case ({\bf Case 2}).  
Finally, Section~\ref{sec:NonPoincareSemiGroup} is devoted to the proof of Theorem~\ref{th:MAIN} in the general case.  

\medskip\noindent
{\bf Acknowledgements.}   
The second author's work is supported by the
french ``ANR blanche'' project Stab: ANR-12-BS01-0019. We thank M.~Hairer and A.~Guillin for enlightening discussions on several results in relation with the present work.

\bigskip

\section{The splitting $\LL = \AA + \BB$ and growth estimates on $S_\BB$ }
\label{sec:bornesB}
\setcounter{equation}{0}
\setcounter{theo}{0}

 We introduce the splitting of the operator $\LL$ defined by  
\begin{equation}\label{eq:FPdefB}
\AA f := M \chi_R f, \qquad \BB f := \LL f - M \chi_R f
\end{equation}
where $M$ is positive constant, and for a fixed truncation function $\chi \in {\mathcal D}({\Bbb R}^d)$ such that $1_{B(0,1)} \leq \chi \leq 1_{B(0,2)}$, and for $R > 1$ which will be chosen appropriately as well as $M$, we set $\chi_R(x) := \chi(x/R)$.  We establish several growth and regularity estimates on the semigroup $S_\BB$ and the family of operators $\AA S_\BB$ which will be of fundamental importance in the sequel. 

\subsection{Basic growth estimates}

\begin{lem}\label{lem:SBLpmLp} For any exponent  $p \in [1,\infty]$ and any  polynomial or  exponential weight function $m$ given by \eqref{eq:def:weightpoly}, \eqref{eq:def:weightexpo1} or \eqref{eq:def:weightexpo2},  we can choose $R,M$ large enough in the definition \eqref{eq:FPdefB} of $\BB$
such that the operator $\BB$ is dissipative in $L^p(m)$, namely
\beqn\label{eq:SBLpcontract}
\|  S_\BB(t) \|_{L^p(m) \to L^p(m)} \leq 1, \qquad \forall \, t \ge 0.
\eeqn
Moreover, if $m(x) = \xx^k$, set $\beta := k(1-\theta)/(2-\gamma)$ for $k^*/k < \theta < 1$, and $\beta \in (0, (k-k^*)/(2-\gamma))$ arbitrary when $\theta \leq k^*/k$. Then the function $\Theta_{m}$ being defined by \eqref{eq:MainEstimPoly},
we have
\beqn\label{eq:SBLpmLp}
\|S_\BB(t) \|_{L^p(m) \to L^p(m^\theta)} \lesssim  \Theta_{m}(t).
\eeqn
If $m(x) = \exp(\kappa\xx^s)$ satisfies \eqref{eq:def:weightexpo1} or \eqref{eq:def:weightexpo2}, 
 the above inequality holds, provided the function $\Theta_{m}$ is defined by
$$\Theta_{m}(t) := \exp(-\lambda t^{s/(2-\gamma)}),$$
where $\lambda > 0$ can be chosen arbitrarily when $s < \gamma$, and $\lambda < \lambda_{*}$, with
$$\lambda_{*} := (\kappa(1-\theta))^{(2 - 2\gamma)/(2 - \gamma)}(\kappa\gamma(1 - \kappa\gamma))^{\gamma/(2-\gamma)},$$
when $s = \gamma$.
\end{lem}

\begin{proof}[Proof of Lemma~\ref{lem:SBLpmLp}.] The proof is similar to the proof of \cite[Lemma 3.8]{GMM}
and \cite[Lemma 3.8]{MM*}.

\smallskip\noindent
{\bf Step 1. } We first  fix $p \in [1,\infty)$, assuming $m$ is as in the statement of the Lemma. 
We start recalling an identity satisfied by the operator $\BB$ (see the proof of \cite[Lemma 3.8]{MM*}). 
For any smooth, rapidly decaying  and positive function $f$, we have  
\bean
\label{eq:FPhomo-BLp}
&& \int_{\R^d} (\BB \, f) \, f^{p-1}  \,  m^p  d x   =
 \\ 
 &&\qquad= 
 - (p-1) \int_{\R^d}  |\nabla (mf)|^2 \, (mf)^{p-2} d x    +
 \int_{\R^d}  f^p \, m^p \, \psi^0_{m,p}(x) d x  ,
\eean
with
\begin{equation}\label{eq:Def-psi-mp}
\psi^0_{m,p}(x) := {(2 - p) \over p}{\Delta m \over m}  +
{2 \over p'}  {|\nabla m|^2 \over m^2} + {1 \over p'} \, {\rm div}({\bf F}) - {\bf F} \cdot {\nabla m \over m} - M \,
\chi_R.
\end{equation}
Observe that 
\bean
&&{\nabla m  \over m} = k \kappa  x  \langle x \rangle^{s-2} 
\\
&&{\Delta m \over m}  =   k \kappa d  \langle x \rangle^{s-2} +  s (s-2) \kappa |x|^2 \langle x \rangle^{s-4} 
+  \nu |x|^2 \langle x \rangle^{2s-4} , 
\eean
where we have set 
\begin{align*}
s &:= 0, \qquad \kappa := 1, \qquad \nu := k(k-2), \qquad \mbox{when }\, m(x) = \langle x \rangle^k, \\
k &:= s,\qquad \nu := (s\kappa)^2, \qquad\qquad\qquad\qquad \mbox{when }\,
m(x) = \exp(\kappa \, \langle x \rangle^s).
\end{align*}
In this latter case, for $s \in (0,\gamma]$, the third term in the definition of $\psi^0_{m,p}$ is negligible with respect to the first and second terms,  and thus 
\bean 
\psi^0_{m,p}(x)\, | x |^{2 - \gamma-s}  &&\!\!\!\!\mathop{\longrightarrow}_{|x| \to \infty} - a^* := (\kappa\gamma)^2 - \kappa\gamma < 0 \quad \mbox{if }\, s=\gamma, 
\\
\psi^0_{m,p}(x)\, | x |^{2 - \gamma-s}  &&\!\!\!\!\mathop{\longrightarrow}_{|x| \to \infty} - a^* :=  -\infty \quad \mbox{if }\, 0< s < \gamma.
\eean

When $m = \langle x \rangle^k$, and $k > k^*(p) := C_{F}/p'$,  the first and second terms are  negligible with respect to the third term, 
and then 
\bean 
\limsup_{|x|  \to \infty}  \psi^0_{m,p}(x) \, | x |^{2 - \gamma} \le - a^* :=   (1 - {1 \over p}) \, C_F - k < 0. 
\eean

We deduce that for any $a \in (0,a^*(m,p))$, we can choose  $R > 1$ and $M$ large enough in such a way that 
$\psi^0_{m,p}(x)  \le - a \langle x\rangle^{\gamma+s-2}$ for all $x \in \R^d$, and then 
\begin{equation}\label{eq:BdissipLp}
\int (\BB f) \, f^{p-1} \, m^p \le  -a \int |f|^p \, m^p \, \langle x \rangle^{\gamma+s-2} -  (p-1) \int |\nabla (fm)|^2 (fm)^{p-1} .
\end{equation}
In particular, using only the fact that the RHS term is negative, we conclude that the operator $\BB$ is dissipative and we classically deduce that the semigroup $S_\BB$ is well-defined on $L^p(m)$ for $p \in [1,\infty)$ and that it is a strongly continuous contraction semigroup, in other words, \eqref{eq:SBLpcontract} holds for any $p \in [1,\infty)$. 
Since we may choose $R,M$ such that the above inequality holds true for any $p \in [1,\infty)$ when $a \in (0,a^*(m,\infty))$, we may pass to the limit as $p\to\infty$ in \eqref{eq:SBLpcontract} and we conclude that $S_\BB$ is a contraction semigroup in $L^p(m)$,  for  any $p \in [1,\infty]$.

\smallskip\noindent
\noindent{\bf Step 2. } Take $p \in [1,\infty)$ and $k > k^*(p) = C_{F}/p'$, and finally, assuming first that $\theta > k^*/k$, set $\ell := \theta k \in (k^*,k)$. If $f_0 \in L^p(m)$ with $m := \langle x \rangle^k$,  denote $f(t) := S_\BB(t) f_0$. 
Dropping the  last term in \eqref{eq:BdissipLp},  we have for $a \in (0, a^*(m,p))$
$$
{d \over dt}  \int|f|^p \,  \,\langle x \rangle^{p\ell} \le - ap   \int|f|^p \,\langle x \rangle^{p\ell+\gamma-2}.  
$$
Using H\"older's inequality 
$$
\int f^p \, \langle x \rangle^{p\ell} 
\le \Bigl( \int f^p \, \langle x \rangle^{p\ell+\gamma-2} \Bigr)^\eta 
 \Bigl( \int f^p \, \langle x \rangle^{pk} \Bigr)^{1-\eta}
$$
with $\eta := (k-\ell)/[k-\ell + (2-\gamma)/p] \in (0,1)$, and the fact that the semigroup $S_{\BB}$ is a contraction semigroup in $L^p_{k}$ by \eqref{eq:SBLpcontract}, upon denoting $\alpha := \eta/(1-\eta) = p \, (k-\ell)/(2-\gamma) $, we get 
$$
{d \over dt}  Y_{\theta} (t) \le - a\, p  \, Y_{\theta}(t)^{(\alpha + 1)/\alpha} \, Y_{1}(0)^{-1/\alpha}, \qquad
\mbox{where }\, Y_{\tau}(t) := \int f^p \xx^{p\tau}.
$$
Integrating the above differential inequality yields 
$$
Y_{\theta} (t) \le \bigl( {\alpha \over a p t} \bigr)^\alpha \, Y_{1}(0),
$$
which in turn implies \eqref{eq:SBLpmLp} with $\Theta_{m} (t)$ replaced with
$
\left( { (k-\ell)/(2-\gamma) \over at} \right)^{{  k-\ell \over 2-\gamma  }}.
$
Since for $0\leq t \leq 1$ we have clearly $Y_{\theta}(t) \lesssim Y_{1}(0)$ the proof of \eqref{eq:SBLpmLp} is complete when $p < \infty$ and $m(x) = \xx^k$ and $\ell := k\theta > k^*$. 

In the case where $\ell = k\theta \leq k^*$, it is enough to pick $\theta_{0} > \theta $ so that $k\theta_{0} > k^*$ and observe that we have $Y_{\theta} \leq Y_{\theta_{0}}$: in this way one is convinced that \eqref{eq:SBLpmLp} holds for all $p<\infty$ and $0 \leq \theta < 1$.  

We deduce the same estimate for $p=\infty$ by letting $p\to\infty$ in \eqref{eq:SBLpmLp}.

\smallskip
\noindent{\bf Step 3.} Similarly, when the weight function $m$ is an exponential as defined in \eqref{eq:def:weightexpo1} or \eqref{eq:def:weightexpo2},
take $p \in [1,\infty)$. Given an initial datum $f_0 \in L^p(m)$, denote $f(t) := S_\BB(t) f_0$, and set
$Y_{\theta}(t) := \|  f(t) \|_{L^p(m^{\theta})}^p$. Thanks to the above Step 1 we have for all $t \geq 0$ and  $0 < \theta \leq 1$
$$
Y_{\theta}(t) \le Y_{\theta}(0).
$$
For $\rho >0$ denote by $B_{\rho}$ the ball of ${\Bbb R}^d$ centered at the origin with radius $\rho$. Using the estimate \eqref{eq:BdissipLp} with the weight function $m^{\theta}$, neglecting the last term of that inequality,  we have successively 
\begin{align*}
{d \over dt}  Y_{\theta} (t)
&= p \int (\BB f) \, f^{p-1}  \, m^{p\theta}
\\
&\le -a\,p \int_{B_\rho}  |f|^p   \, m^{p\theta} \, \langle x \rangle^{\gamma+s-2}  
\\
&\le -a\,p\, \langle\rho\rangle^{\gamma+s-2} \int_{B_\rho}  |f|^p \, m^{p\theta}
\\
&\le -a\,p\, \langle\rho\rangle^{\gamma+s-2}    Y_{\theta}  + 
a\,p\, \langle\rho\rangle^{\gamma+s-2}    \int_{B^c_\rho}  |f|^p   \, m^{p\theta}   
\\
&\le -a\,p\, \langle\rho\rangle^{\gamma+s-2}    Y_{\theta}  + 
a\,p \,\langle\rho\rangle^{\gamma+s-2} m(\rho)^{-p(1-\theta)} \int_{B^c_{\rho}} |f_0|^p   \, m^p   \\
&\le -a\,p\, \langle\rho\rangle^{\gamma+s-2} Y_{\theta}  + 
a\,p \,\langle\rho\rangle^{\gamma+s-2}\, m(\rho)^{-p(1-\theta)}\int   |f_0|^p   \, m^p .
\end{align*}
Integrating this differential inequality we deduce 
\begin{align*}
Y_{\theta}(t) & \le \exp(-ap\,t\,\langle\rho\rangle^{\gamma+s-2})\,  Y_{\theta}(0) +    m(\rho)^{-p(1-\theta)}\, Y_{\theta}(0).
\\
&\le 
 \left(\exp(-ap\,t\,\langle\rho\rangle^{\gamma+s-2})  + 
 \exp(-p(1-\theta)\,\rho^s)\right)\, Y_{\theta}(0).
\end{align*}
We may choose $\rho$ such that $a\,\langle\rho\rangle^{\gamma+s-2}t =   (1 - \theta)\,\rho^s$, that is we may take $\rho$ of order $t^{1/(2-\gamma)}$, which allows us to conclude that \eqref{eq:SBLpmLp} also holds in the  exponential case.
As indicated above, the estimate \eqref{eq:SBLpmLp} for $p=\infty$ is obtained by letting $p\to\infty$. 
 \end{proof}
\medskip

The following two lemmas state that when the weight function is  exponential, that is $m(x):=\exp(\kappa\xx^\gamma)$, the semigroup $S_{\BB}$ is ultracontractive in the spaces $L^p(m)$, that is it maps $L^1(m)$ into $L^\infty(m)$ for $t > 0$. As it is pointed out in \cite{OK-KR} (see Remark 2.2 of this reference for a proof based on Probability arguments, and Remark 5.2 for a simple proof based on comparison theorems for parabolic equations), when one considers an operator of the type $Lf := \Delta f + \nabla V \cdot \nabla f$ with $V$ satisfying, for some constants $R > 0$ and $c_{0} \geq 0$, 
$$ \forall x \in B_{R}^c, \qquad\qquad
{\Delta V^{1/2} \over V^{1/2}} + c_{0} \geq 0,
$$
and if there exists a positive constant $c_{1} > 0$ such that $V(x) \geq c_{1}$ for all $x\in B_{R}^c$, then the semigroups $S_{L}(t)$ and $S_{L^*}(t)$ are ultracontractive in the spaces $L^p(\exp(V))$ (the above condition on $\Delta V^{1/2}/V^{1/2}$ is a sort of convexity condition at infinity). Here the operator $\BB$ is not exactly of the same type as $L$, but nevertheless the ultracontractivity of the semigroup $S_{\BB}(t)$, as well as that of $(S_{\BB})^*$, holds. (Recall also that the ultracontractivity of the semigroup $S_{\BB}$ is equivalent to an appropriate form of Nash inequality for the operator $\BB$).

\begin{lem}\label{lem:Ultracon-1}
Consider the weight function  $m_0 := \exp(\kappa \langle x \rangle^\gamma)$, for $0 < \kappa \gamma < 1$. Then there exists $R_{0}, M_{0} > 0$ such that for $M \geq M_{0}$ and $R \geq R_{0}$ we have
\beqn\label{eq:SBL1L2}
\forall \, t   > 0, \qquad \qquad \|  S_\BB(t) \|_{L^1(m_0) \to L^2(m_0)} \lesssim  t^{-d/4} .
\eeqn
\end{lem}

\begin{proof}[Proof of Lemma~\ref{lem:Ultracon-1}.]
The proof is similar to the proof of \cite[Lemma 3.9]{GMM}, see also 
\cite[Section 3]{MM*}, \cite[Section 2]{MQT*} and \cite{OK-Kompaneets}. For the sake of completeness we sketch it below. 

Consider $f_0 \in L^2(m_0)$ and denote $f(t) = S_\BB(t) f_0$.  From \eqref{eq:FPhomo-BLp}
with  $p=2$ and throwing out the last term in that inequality, we find 
$$
{d \over dt} {1 \over 2} \int_{\R^d} f(t)^2   \, m_0^2 dx \le  - \int_{\R^d} |\nabla ( f(t) m_0) |^2 d x. 
$$
Using Nash's inequality for $g:=f(t)m_{0}$ (\cite[Chapter~8]{MR1817225}) stating that for some constant $c > 0$ 
\beqn\label{eq:NashIneq}
  \int_{\R^d} g^2 dx  \le c \, \left( \int_{\R^d} |
  \nabla  g |^2 d x \right)^{\frac{d}{d+2}} \, \left( \int_{\R^d} |g|
  d x \right)^{\frac{4}{d+2}}
\eeqn
we get (for another constant $c > 0$)
\begin{equation}\label{eq:Nash0}
X'(t) \le    - 2 \, c\, Y(t)^{-4/d} \, X(t)^{1+{2\over d}},
\end{equation}
where for brevity of notations we have set
$$
X(t) := \|f(t) \|_{L^2(m_0)}^2, \qquad Y(t) := \| f(t) \|_{L^1(m_0)}.
$$
Since according to \eqref{eq:SBLpcontract} we have  $Y(t)\le Y_0$ for $t > 0$, we may integrate the differential inequality \eqref{eq:Nash0} and obtain \eqref{eq:SBL1L2}. 
\end{proof}

The next result states that the adjoint of $\BB$ generates also an ultracontractive semigroup in the spaces $L^p(m_{0})$.

\begin{lem}\label{lem:Ultracon-2}
Consider the weight function  $m_0 := \exp(\kappa \langle x \rangle^\gamma)$, for $0 < \kappa \gamma < 1$. Then there exists $R_{1} \geq R_{0}$ and $M_{1} \geq M_{0}$ (where $M_0$ and $R_0$ are defined in the previous lemma) such that for $M \geq M_{1}$ and $R \geq R_{1}$, the semigroup generated by $\BB_{*}$,  the formal adjoint of $\BB$, satisfies
\beqn\label{eq:SB*L1L2}
\forall\, t >0, \qquad 
\|S_{\BB_{*}}(t) \|_{L^1(m_0) \to L^2(m_0)} \lesssim t^{-d/4}.
\eeqn
Consequently, for $M \geq M_{1}$ and $R \geq R_{1}$, we have
\beqn\label{eq:SBL2inf}
\forall \, t   > 0, \qquad \qquad \|  S_\BB(t) \|_{L^2(m_0) \to L^\infty(m_0)} \lesssim  t^{-d/4} .
\eeqn
\end{lem}

\begin{proof}[Proof of Lemma~\ref{lem:Ultracon-2}.]
We first observe that if the operator $B$ is of the form
$$
B f = \Delta f +  {\bf b}(x) \cdot \nabla f + a(x) \, f,
$$
and we make the transform $h := f \, m$, then the corresponding operator $B_m h := m \, B(m^{-1}\, h)$ is of the same type and is given by 
\begin{equation*}
B_m h = \Delta h + 
\left[{\bf b}(x) - 2 \, {\nabla m \over m} \right]\cdot \nabla h  + 
 \left[- {\Delta m \over m} + 2 \, {|\nabla m|^2 \over m^2} + a - {\bf b}(x) \cdot {\nabla m \over m} \right] \, h.
\end{equation*}
Observe also that the formal adjoint of $B$, denoted by $B_{*}$ to avoid any misunderstanding, is given by 
$$
B_{*} g  = \Delta g - {\bf b}(x) \cdot \nabla g + 
 (a(x) - {\rm div}({\bf b}(x))) \, g.
$$
Applying these observations to  
$
\BB f = \Delta f  + {\bf F} \cdot \nabla f +  ({\rm div} ({\bf F}) - M \chi_R) \, f
$,
we get that for $h := g\, m_{0}$ the operator $\BB_{*,m}$, associated to the formal adjoint $\BB_{*}$, is given by
\begin{equation} \label{eq:defB*m}
\BB_{*,m} h = \Delta h - 
\left[{\bf F} - 2 \, {\nabla m \over m} \right]\cdot \nabla h  + \left[ {\Delta m \over m}   - M \chi_R - {\bf F} \cdot {\nabla m \over m}  \right] \, h.
\end{equation}
Thus, if $g_{0} \geq 0$ is a smooth initial datum, then the solution $g$ of 
$$\partial_{t}g = \BB_{*}g, \qquad g(0,x) = g_{0}(x),$$
yields the function $h := g\, m_{0}$ which satisfies the evolution equation
\begin{equation}\label{eq:eqSBm*}
\partial_t h = \BB_{*,m_0} h, \qquad h(0,x) = h_{0}(x).
\end{equation}
Now, one can verify easily that for $h_{0} \in C^\infty_{c}({\Bbb R}^d)$ and $h_{0} \geq 0$, the solution $h$ to the equation 
$$
\partial_t h = \Delta h + {\bf b}(x) \cdot \nabla h + a(x)\,h, \qquad h(0,x) = h_{0}(x)
$$
satisfies $h(t,x) \geq 0$ and for $ 1 \leq p < \infty$ we have the identity
\begin{equation*}
{d \over dt} {1 \over p}\int h(t,x)^p\,dx  = - (p-1)  \int |\nabla h|^2 \,h^{p-2}\,dx  + {1 \over p}\int (p\, a(x) - {\rm div}({\bf b}(x))) \, h^p\, dx . 
\end{equation*}
As a consequence, applying this to the operator $\BB_{*,m_{0}}$, we have that the solution $h$ of equation \eqref{eq:eqSBm*} verifies
\begin{equation*}
{d \over dt}{1 \over p}\int h(t,x)^p \, dx \leq  - 
(p-1)\int |\nabla h|^2h^{p-2}\,dx +  
\int h^p \, \psi_{*,m_{0},p}\,dx,
\end{equation*}
where, for convenience, we have set 
\begin{equation}\label{eq:def-phipm}
\psi_{*,p,m_0} :=  
 {(p - 2) \over p}) {\Delta m_0 \over m_0}  +  
 {2 \over p} \,  {|\nabla m_0|^2 \over m_0^2} +
 {1 \over p}\, {\rm div}({\bf F}) - {\bf F} \cdot {\nabla m_0 \over m_0}  - M \chi_R  . 
\end{equation}
Proceeding as we did above in the study of the function defined in \eqref{eq:Def-psi-mp}, we may choose, if necessary, $M$ and $R$ large enough (in particular larger than $M_{0},R_{0}$ given by Lemma \ref{lem:Ultracon-1}), so that for all $x \in {\Bbb R}^d$ we have $\psi_{*,m_{0},} \leq 0$. Therefore we conclude that 
\begin{equation}\label{eq:phiLp}
 {d \over dt} {1 \over p}   \|h(t)\|^p_{L^p}  \le -  (p - 1) \, \int |\nabla h|^2h^{p-2}\,dx.
\end{equation}  
On the one hand, taking $p := 1$, we deduce that the semigroup generated by $\BB_{*,m_{0}}$ is a contraction semigroup in $L^1({\Bbb R}^d)$, 
that is $\|h(t)\|_{L^1} \leq \|h_{0}\|_{L^1}$  for all $t \ge 0$.

On the other hand, taking $p :=2 $ and using  Nash's inequality \eqref{eq:NashIneq}, together with the fact that $\|h(t)\|_{L^1}$ is non increasing, we deduce that if we set $X(t) := \|h(t) \|_{L^2}^2$, then for some constant $c >0 $, the function $X(t)$ satisfies the differential inequality
$$
{d \over dt} X(t) \leq - c\, \|h_{0}\|_{L^1}^{-4/d} \, X(t)^{(2+d)/d}. 
$$ 
Integrating this,  we get that for all $t > 0$
$$
\|h(t)\|_{L^2} \lesssim t^{-d/4}\,\|h_{0}\|_{L^1}.
$$
From this, by a density argument and the splitting of any initial datum as the difference of two nonnegative functions,  we conclude that for any $g_{0} \in L^1(m_{0})$  the associated solution to $\partial_{t} g = \BB_{*}g$ satisfies
\begin{equation}\label{eq:SB*-Ultracon}
\|S_{\BB_{*}}(t)g_{0}\|_{L^2(m_{0})} = \|g(t)\|_{L^2(m_{0})} \lesssim t^{-d/4}\, \|g_{0}\|_{L^1(m_{0})},  \quad \forall t > 0,
\end{equation}
which is precisely \eqref{eq:SB*L1L2}.
To conclude the proof of the Lemma, observe that for $f,g\in C^\infty_{c}({\Bbb R}^d)$ we have
\begin{align*}
(\BB f|g)_{L^2(m_{0})} &= \int \BB f\, g\, m_{0}^2\,dx =
\int \BB_{m_{0}}(f\,m_{0})\, (g\,m_{0})\,dx\\
&= \int (f\,m_{0})\, \BB_{*,m_{0}}(g\,m_{0})\,dx =
\int (f\,m_{0})\, m_{0}^{-1}\BB_{*,m_{0}}(g\,m_{0})\, m_{0}(x)^2\,dx\\
&= (f|\BB_{*}g)_{L^2(m_{0})}.
\end{align*}
This allows one to verify that $(S_{\BB}(t))^* = S_{\BB_{*}}(t)$, the adjoint being taken in the sense of the Hilbert space $L^2(m_{0})$, where we assume that this space is identified with its dual. Therefore, since with these conventions we have  $(L^1(m_{0}))' = L^\infty(m_{0})$, thanks to \eqref{eq:SB*-Ultracon}, we conclude \eqref{eq:SBL2inf}. 
\end{proof}

Putting together the previous estimates, we get the following ultracontractivity result on the semigroup $S_\BB$ 
and on the iterated convolution family of operators $(\AA S_\BB)^{(*n)}$. 

\begin{lem}\label{lem:SBLpLq} 
Consider the weight function  $m_0 := \exp(\kappa \langle x \rangle^\gamma)$, for $0 < \kappa \gamma < 1$.  
Then, $M,R$ being large enough as in Lemma \ref{lem:Ultracon-2}, there exists $\lambda_{*} \in (0,\infty)$ such that for any $p,q \in [1,\infty]$, $p\leq q$, and for any $0 \leq \theta_{2} < \theta_{1} \leq 1$, the semigroup $S_\BB$ satisfies  
\beqn\label{eq:SBLpLq-theta}
\|S_\BB(t)\|_{L^p(m_0^{\theta_{2}}) \to L^q(m_{0}^{\theta_{1}})} \lesssim  t^{-(d/2) (1/p - 1/q)}  \, {\rm e}^{- \lambda_{*} \, t^{\gamma/(2-\gamma)} },
\quad \forall \, t > 0.
\eeqn
Moreover, if $n \geq d/2 $ is an integer, for all $\lambda < \lambda_{*}$ and all $t > 0$, we have
 \beqn\label{eq:ASBnL1Linfty}
\|(\AA S_\BB)^{(*n)}(t) \|_{L^1(m_0) \to L^\infty(m^2_0)} \lesssim {\rm e}^{- \lambda \, t^{\gamma/(2-\gamma)}} .
\eeqn
\end{lem}

\begin{proof}[Proof of Lemma~\ref{lem:SBLpLq}.]
{\sl Step 1. } 
Writing $S_{\BB}(t) = S_{\BB}(t/2)S_{\BB}(t/2)$, and using \eqref{eq:SBL1L2} together with \eqref{eq:SBL2inf}, we deduce that for any $t > 0$  
\begin{align}\label{eq:SBL1Linfty}
\|S_\BB(t) \|_{L^1(m_0) \to L^\infty(m_0)} &\leq 
\|S_\BB(t/2) \|_{L^2(m_0) \to L^\infty(m_0)}\,
\|S_\BB(t/2) \|_{L^1(m_0) \to L^2(m_0)} \nonumber\\
&\lesssim t^{-d/2}. 
\end{align}
Since on the other hand we have also $\|S_{\BB}(t)f_{0}\|_{L^p(m_{0})} \leq \|f_{0}\|_{L^p(m_{0})}$, a classical interpolation argument yields that for  $1 \le p \le q \le  \infty$, we have 
 \beqn\label{eq:SBLpLqbis}
\|S_\BB(t) \|_{L^p(m_0) \to L^q(m_0)} \lesssim t^{-(d/2) (1/p - 1/q)}.
\eeqn
Using Lemma \ref{lem:SBLpmLp}, since 
$$\|S_{\BB}(t)\|_{L^q(m_{0}) \to L^q(m_{0}^\theta)} \lesssim 
\exp(-\lambda t^{\gamma/(2-\gamma)}),$$ one sees that the proof of \eqref{eq:SBLpLq-theta} with $\theta_1 = 1$ is complete. To see that \eqref{eq:SBLpLq-theta} holds with $\theta_1 < 1$, it is enough to observe that $m_{0}^{\theta_{1}}(x)$ is of the same type as $m_{0}(x) = \exp(\kappa\xx^\gamma)$ provided $\kappa$ is replaced with $\theta_{1}\kappa$.
\medskip

\noindent {\sl Step 2. } 
In order to show \eqref{eq:ASBnL1Linfty}, first note that the operator $\AA$ consisting simply in a multiplication by a smooth compactly supported function, thanks to the above lemmas, we clearly have, for all $t > 0$, 
\beqn\label{eq:ASBL1Linfty}
\|  \AA S_\BB(t) \|_{L^{p_1}(m_0) \to L^{p_2}(m^2_0)} \lesssim  t^{-\alpha} \, {\rm e}^{-\lambda \, t^{\sigma^*}},
\eeqn
with $\sigma^* := \gamma/(2-\gamma)  \leq 1$ and 
where it is understood that $\alpha := d/2$ if $(p_{1},p_{2}) := (1,\infty)$, and $\alpha := 0$ when $p_{1} = p_{2}$. We claim that the  three estimates for three choices $(p_{1},p_{2}) = (1,\infty)$, and $p_{1} = p_{2} = 1$, as well as $p_{1} = p_{2} = \infty$, imply that for all integers $n \geq 1$ we have
\beqn\label{eq:ASBnL1LinftyBIS}
\|(\AA S_\BB)^{(*n)}(t) \|_{L^{p_1}(m_0) \to L^{p_2}(m^2_0)} \leq C_{n}\, t^{n-1-\alpha}   \, {\rm e}^{-\lambda \, t^{\sigma^*}},
\eeqn
from which one readily deduces \eqref{eq:ASBnL1Linfty}. We prove \eqref{eq:ASBnL1LinftyBIS} by induction.
Estimates  \eqref{eq:ASBnL1LinftyBIS} are clearly true for $n=1$. Let us assume that $p_{1} = 1$ and $p_{2} = \infty$, for which \eqref{eq:ASBnL1LinftyBIS} is true for a certain $n \ge 1$.  
Introducing the shorthand notation $u := \AA S_\BB$ and $\| \cdot \|_{p_1\to p_2} = \|  \cdot \|_{L^{p_1}(m_0) \to L^{p_2}(m^2_0)}$, we have 
\begin{align*}
\|u^{(*(n+1))} (t) \|_{1\to\infty} 
&\le \int_0^{t/2} \| u^{(n)}(t-s) \|_{1\to\infty} \, \| u(s) \|_{1\to1} \, ds 
\\
& \quad + \int_{t/2}^t \| u^{(n)}(t-s) \|_{\infty\to\infty} \, \| u(s) \|_{1\to\infty} \, ds 
\\
&\le C_n \, C_1\,  {\rm e}^{-\lambda \, t^{\sigma^*}}  \int_0^{t/2} (t-s)^{-\alpha + n-1} \, ds 
\\
& + C_n \, C_1 \,  {\rm e}^{-\lambda \, t^{\sigma^*}}   \int_{t/2}^t  (t-s)^{n-1}  \,  s^{-\alpha }  \, ds
\\
&\leq \quad  C_n \, C_1\,  {\rm e}^{-\lambda \, t^{\sigma^*}} t^{-\alpha + n} \left\{\int_0^{1/2} (1-\tau)^{-\alpha+n-1} \, d\tau   \right.\\ 
& \qquad\qquad \left. + \int_{1/2}^1  (1-\tau)^{n-1}  \, \tau^{-\alpha }  \, d\tau \right\} ,
\end{align*}
where we have used the fact that $t^{\sigma^*} \le (t-s)^{\sigma^*} + s^{\sigma^*}$ for any $0 \le s \le t$, since $0 < \sigma^* \leq 1$. This proves estimate \eqref{eq:ASBnL1LinftyBIS} at rank $n+1$ and $(p_1,p_2) = (1,\infty)$. 

The proof of the other cases  $(p_1,p_2) = (1,1)$ and $(p_1,p_2) = (\infty,\infty)$ is similar, if not much simpler, and can be left to the reader. 
\end{proof}

\subsection{Additional  growth estimates} In order to deal with the general case in Section~\ref{sec:NonPoincareSemiGroup},  we will need a more accurate version of the previous estimates. 

\begin{lem}\label{lem:SBL2weight} Consider $m_0 := {\rm e}^{2\kappa \langle x \rangle^\gamma}$ with $\kappa \in (0,1/(4\gamma))$  and
define the sequence of spaces 
\beqn\label{eq:defXk}
X_k := L^2(m_k), \quad m_k  := {m_0   \over \nu_k}, \quad \nu_k(x) :=  \sum_{\ell=0}^k {(\kappa\,  \langle x  \rangle^\gamma)^\ell \over  \ell!  }, 
\eeqn
for any $k \in \N$.
There exist some constants $R$ and $M$ in the definition of $\BB$  and  some constant $\beta > 0$ such that for any $k, j \in \N$,
$k \ge j$ and any $\alpha \in (0,\alpha^*)$,  $\alpha^* := 1/2(1-\gamma)$, the semigroup $S_\BB$ satisfies the growth estimate 
\beqn\label{eq:SBXk-jXk}
\|  S_\BB(t) \|_{X_{k-j} \to X_k} \lesssim  
 {\rm e}^{-  \lambda \, \langle t^\alpha \rangle^{2(\gamma-1)} t}  + \Bigl( 1 \wedge  {k^{j} \over  \kappa^j \, \langle t^\alpha \rangle^{\gamma j}}  \Bigr) \quad \forall \, t > 0.
\eeqn
 
\end{lem}

\begin{proof}[Proof of Lemma~\ref{lem:SBL2weight}.]
We easily compute 
$$
{\nabla m_k \over m_k}  = \gamma \, \kappa \, x \, \langle x \rangle^{\gamma - 2}  \, 
\Bigl[  2    - { 1+ ... + (\kappa\,  \langle x  \rangle^\gamma)^{k-1} / (k-1)! \over 1+ ... + (\kappa\,  \langle x  \rangle^\gamma)^k / k! } \Bigr]   ,
$$
from which we deduce 
$$
\Bigl| {\nabla m_k \over m_k}\Bigr|   \le 2 \,  \gamma \, \kappa \, |x| \, \langle x \rangle^{\gamma - 2}    ,
$$
and
$$
-  {\bf F} \cdot {\nabla m_k \over m_k}    \le  -  \gamma \, \kappa \, |x|^\gamma \, \langle x \rangle^{\gamma - 2}   \quad \forall \, x \in B_{R_0}^c.
$$
As a consequence, for $f(t) := S_{\BB}(t)f_{0}$, we have 
$$
{d \over dt}  \int f^2 \, m_k^2  \le - \int |\nabla f|^2 \, m_k^2 + \int f^2 \, m_k^2  \, \psi_k,
$$
with
\bean
\psi_k 
&=&  {|\nabla m_k|^2 \over m_k^2} +
{1 \over 2}  \, {\rm div}({\bf F}) - {\bf F} \cdot {\nabla m_k \over m_k} - M \, \chi_R
\\
&\le&  4 \,  \gamma^2 \, \kappa^2 \, |x|^2 \, \langle x \rangle^{2\gamma - 4} + {1 \over 2}  \, C'_F \, \langle x \rangle^{\gamma-2} 
\\
&&
\qquad -  \gamma \, \kappa \, |x|^\gamma \, \langle x \rangle^{\gamma - 2} \, {\bf 1}_{B_{R_0}^c} - M \, \chi_R
\\
&\le& -  2\lambda  \, \langle x \rangle^{2(\gamma - 1)}
\eean
for any $x \in \R^d$ and $k \in \N$,  by fixing $\kappa > 0$ small enough (as we did) and then $R$  and $M$ large enough.  We deduce 
\beqn\label{eq:dtYk}
{d \over dt}  \int f^2 \, m_k^2  \le - \int |\nabla f|^2 \, m_k^2  - 2\lambda \, \int f^2 \, m_k^2 \,  \langle x \rangle^{2(\gamma - 1)}, 
\eeqn
and in particular 
$$
Y_k(t) :=    \int f^2 \, m_k^2  \le Y_k(0) \quad\mbox{for any}\quad  k \ge 0.
$$
We now observe that  for any $j \in \N$, $0 \le j \le k$, there hold $m_k \le m_{k-j}$ as well as  
$$
m_k(x) 
\le {m_0 (x)\over   \nu_{k-j}(x) \,( \kappa \, \langle x \rangle^\gamma /k)^j}  
= {k^{j} \over  \kappa^j \, \langle x \rangle^{\gamma j}}  \, m_{k-j}(x) \quad \forall \, x \in \R^d.
$$
The two inequalities together, we have proved 
\beqn\label{eq:mkmkj}
\forall \, j \le k, \,\forall \, x \in \R^d \quad m_k(x) \le \Bigl( 1 \wedge  {k^{j} \over  \kappa^j \, \langle x \rangle^{\gamma j}}  \Bigr) \, m_{k-j}(x).
\eeqn
As a consequence, for any $\rho > 0$, we have 
\bean
Y_k 
&=& \int_{B_\rho} f^2 \, m_k^2 + \int_{B_\rho^c} f^2 \, m_k^2
\\
&\le& \langle \rho \rangle^{2(1-\gamma)} \int f^2 \, m_k^2  \, \langle x \rangle^{2(\gamma - 1)} +  \Bigl( 1 \wedge  {k^{j} \over  \kappa^j \, \langle \rho \rangle^{\gamma j}}  \Bigr)^2 \, Y_{k-j}.
\eean
Coming back to \eqref{eq:dtYk}, we deduce 
$$
{d \over dt} Y_k \le - 2\lambda \, \langle \rho \rangle^{2(\gamma-1)}  \, Y_k +  2\lambda \, \langle \rho \rangle^{2(\gamma-1)} \Bigl( 1 \wedge  {k^{j} \over  \kappa^j \, \langle \rho \rangle^{\gamma j}}  \Bigr)^2 \, Y_{k-j}(0), 
$$
which in turn implies 
$$
Y_k(t) \le \Bigl\{  {\rm e}^{-  2\lambda \, \langle \rho \rangle^{2(\gamma-1)} t}  + \Bigl( 1 \wedge  {k^{j} \over  \kappa^j \, \langle \rho \rangle^{\gamma j}}  \Bigr)^2 \Bigr\}  Y_{k-j}(0) \quad \forall \, t \ge 0, \, \rho > 0. 
$$
We conclude by making the choice  $\rho = t^{\alpha}$ for $\alpha \in (0,\alpha^*)$. 
\end{proof}

\begin{lem}\label{lem:SBL2H1} Consider $m_0 := {\rm e}^{\kappa \langle x \rangle^\gamma}$ with $\kappa \in (0,1/(2\gamma))$. 
There exist two constants  $C, \lambda > 0$ such that $S_\BB$ satisfies  (recall that $\sigma^* = \gamma/(2-\gamma)$)
\beqn\label{eq:SBL2H1}
 \|  S_\BB(t) \|_{L^2(m_0) \to H^1} \le  {C \over t^{1/2}}  \, {\rm e}^{- \lambda \, t^{\sigma^*}} \quad \forall \, t  >  0.
\eeqn

\end{lem}

\begin{proof}[Proof of Lemma~\ref{lem:SBL2H1}.]
 The function $f = S_\BB(t) f_0$ satisfies  
$$
\partial_t \partial_i f = \Delta \partial_i f + ( \partial_i {\rm div} ({\bf F})) f +  {\rm div}({\bf F}) \, \partial_i f + \partial_i {\bf F}_j \partial_j f + {\bf F} \cdot \nabla \partial_i f - M \partial_i \chi_R f 
- M \chi_R \partial_i f, 
$$
so that 
\bean 
&&{1 \over 2} {d \over dt} \int |\nabla f |^2 m_0^2 
= - \int |D^2 f |^2 m_0^2 + {1 \over 2}  \int |\nabla f |^2 \Delta m_0^2 - \int ({\rm div} {\bf F}) \,  f (\Delta f) \,  m_0^2
\\
&&\qquad 
 - \int ({\rm div}({\bf F})) \,  f \, \nabla f \cdot \nabla m_0^2   - \int \partial_i {\bf F}_j \partial_i f \partial_j f \,  m_0^2
   - {1 \over 2}  \int ({\rm div}({\bf F})) \, | \nabla f |^2  m_0^2
\\
&&\qquad
 - {1 \over 2}  \int | \nabla f |^2  {\bf F} \cdot \nabla  m_0^2 + M \int {\rm div} (\nabla \chi_R m^2_0) f^2  - \int  M \chi_R |\nabla f|^2 
\\
&&\qquad\le - {1 \over 2} \int |D^2 f |^2 m_0^2 dx+ \int |\nabla f|^2 m_0^2  \psi_{1} \, dx+ \int f^2 m_0^2  \psi_{2} \, dx,
\eean
with 
\bean
\psi_{1}(x) := {1  \over 2} {\Delta m^2_0 \over m_0^2} + {1  \over 2} | {\rm div}({\bf F})| {|\nabla m^2_0| \over m_0^2} + {3 \over 2}  |D{\bf F}| - {\bf F} \cdot {\nabla m_0 \over m_0} - M \chi_R 
\eean
and
\bean
\psi_{2}(x) := ({\rm div}({\bf F}))^2 + {1  \over 2} {\Delta m^2_0 \over m_0^2} + {1  \over 2} | {\rm div} {\bf F}| {|\nabla m^2_0| \over m_0^2} +  M  |\Delta \chi_R |  
+ M |\nabla \chi_R| { |\nabla m_0^2|  \over  m^2_0}.
\eean
Choosing $M$ and $R$ large enough, we have for some constants $a > 0$ and $C \in \R$
$$
\psi_{1}(x) \le - a \, \langle x \rangle^{2(\gamma-1)}, \qquad  \qquad 
\psi_{2}(x) \le C  \, \langle x \rangle^{2(\gamma-1)}, 
$$
and, choosing then $\eta > 0$ small enough,  we get
\bean
{d \over dt} \int (f^2 + \eta  |\nabla f |^2)  m_0^2 &\le&  - {1 \over 2} \int (|\nabla  f |^2  + \eta  |D^2 f |^2 )  m_0^2
 \\
&&\qquad  - a  \int (f^2 + \eta  |\nabla f |^2) \,  m_0^2 \, \langle x \rangle^{2(\gamma-1)}. 
 \eean
 On the one hand, keeping only the second term and arguing as in the proof of Lemma~\ref{lem:SBLpmLp}, 
 we get
\beqn\label{eq:SBH1mH1}
\|  S_\BB(t) \|_{\BBB(H^1(m_0),H^1)} \le  \Theta_{m}(t), \quad \forall \, t \ge 0.
\eeqn
On the other hand, using the elementary inequality 
 $$
 \int |\nabla f|^2 m_0^2 = -   \int f D^2 f m_0^2 + {1 \over 2} \int f^2 \Delta m_0^2 \le \|  f \|_{L^2(m_0)} \, \|  f \|_{H^2(m_0)} + C \, \|  f \|^2_{L^2(m_0)},
 $$
 the differential inequality
 $$
{d \over dt}  \|  \nabla f \|_{L^2(m_0)} ^2 \le -  \|  D^2 f \|_{L^2(m_0)} ^2 + \| \psi^2 \|_{L^\infty } \| f \|_{L^2(m_0)} ^2 
$$
and the contraction in $L^2(m_0)$, we then obtain
$$
{d \over dt}  \|  \nabla f \|_{L^2(m_0)} ^2 \le -  \|  f_0 \|_{L^2(m_0)}^{-2} \, \|  \nabla f \|_{L^2(m_0)} ^4 + \| \psi^2 \|_{L^\infty }  \| f _0\|_{L^2(m_0)} ^2 . 
$$
 As in the proof of Lemma~\ref{lem:SBLpmLp}, we deduce
\beqn\label{eq:SBL2mH1m}
\|  \nabla f \|_{L^2(m_0)} ^2 \le {C \over t} \| f_0 \|_{L^2(m_0)} ^2 \quad \forall \, t \in(0,1).
\eeqn
We easily deduce \eqref{eq:SBL2H1} by gathering \eqref{eq:SBH1mH1} and  \eqref{eq:SBL2mH1m}. 
\end{proof}

\section{Boundedness of  $S_\LL$  }
\label{sec:bornesL}
\setcounter{equation}{0}
\setcounter{theo}{0}

In this section, we establish some estimates on $S_\LL$ which are simple results yielded by the iterated Duhamel formula \eqref{eq:Duhamel-gene} and the previous estimates on $S_\BB$.

\begin{lem}\label{lem:LpmLL} For any exponent $p \in [1,\infty]$ and any weight function $m$ given by \eqref{eq:def:weightpoly}, \eqref{eq:def:weightexpo1} or   \eqref{eq:def:weightexpo2},    there exists $C(m,p)$ such that 
$$
\|  S_\LL(t) \|_{\BBB(L^p(m))} \le C(m,p), \quad \forall \, t \ge 0.
$$
\end{lem}

\begin{proof}[Proof of Lemma~\ref{lem:LpmLL}.] The estimate 
\beqn\label{eq:bounfLL1}
\|  S_\LL(t) \|_{\BBB(L^1)} \le 1, \quad \forall \, t \ge 0, 
\eeqn
is clear and is a consequence of the fact that $S_\LL$ is a positive and mass conserving semigroup. 

\smallskip\noindent
{\bf Step 1. } We consider first the case $p=1$ and we write the Duhamel formula
$$
S_\LL = S_\BB   +  S_\BB * \AA S_\LL,
$$
which allows one to deduce that 
\begin{equation*}
\| S_\LL \|_{L^1_k \to L^1_k} \le 
\|S_\BB  \|_{L^1_k \to L^1_k}  + \|S_\BB \|_{L^1_{k+\ell} \to L^1_k}  * \|  \AA S_\LL \|_{L^1 \to L^1_{k+\ell}}    \le C, 
\end{equation*}
because $t\mapsto \|  \AA S_\LL (t)\|_{L^1 \to L^1_{k+\ell}}$ is a bounded function  thanks to \eqref{eq:bounfLL1}, and, upon choosing $\ell := 2 (2-\gamma)$ and applying Lemma \ref{lem:SBLpmLp}, we have $\|S_\BB \|_{L^1_{k+\ell} \to L^1_k} \lesssim \langle t \rangle^{-2}$, showing that $\|S_\BB (t)\|_{L^1_{k+\ell} \to L^1_k}$ belongs to $L^1(0,\infty)$. 

\smallskip\noindent
{\bf Step 2. } In order to study the case $1 < p \leq \infty$, we write the iterated Duhamel formula  \eqref{eq:Duhamel-gene} with $n_{1} := \ell+1$ and $n_{2} := 0$ and we get
$$
S_\LL = S_\BB  + \cdots +   S_\BB * (\AA S_\BB)^{(*(\ell-1))}   +   S_\BB * (\AA S_\BB)^{(*\ell)}  *  (\AA S_\LL)  
$$
with $\ell := d/2+1$ and then we infer that
\begin{align*}
\|  S_\LL \|_{\BBB(L^p(m))} 
&\le  \sum_{j = 0}^{\ell-1} \|S_\BB \|_{\BBB(L^p(m))}  *  \|  \AA S_\BB  \|_{\BBB(L^p(m))}^{(*j)} 
\\
& \hskip -1cm +  \|S_\BB \|_{L^p(\varpi) \to L^p(m)} *   \|  (\AA S_\BB )^{(*\ell)} \|_{L^1(\varpi) \to L^\infty(\varpi^2) } * \|  \AA S_\LL \|_{L^p(m) \to L^1(\varpi)}  ,
\end{align*}
where $\varpi := {\rm e}^{\nu \, \langle x \rangle^\gamma}$  with $\nu \in (0,1/\gamma)$ large enough. 
Using \eqref{eq:ASBL1Linfty} and \eqref{eq:ASBnL1Linfty}, we see that the RHS term is uniformly bounded as the sum of $\ell+1$ functions, each one being the convolution of one $L^\infty$ function with  integrable functions. 
\end{proof}

\begin{rem} 
When $p=1$,  $m = \langle x \rangle^k$ and $f(t) \ge 0$, the above estimate means exactly that for $k > 2-\gamma > 1$, there exists a constant $C_k$ such that 
$$
M_k(S_\LL(t)f_0) \le C_k \, M_k(f_0), \quad M_k(f) := \int_{\R^d} f \, \langle x \rangle^k \, dx. 
$$
We then recover (with a simpler proof) and generalize (to a wider class of confinement force fields) a result obtained by Toscani and Villani in \cite[section 2]{MR1751701}. 
\end{rem}


\section{The {\bf Case 1} when a Poincar\'e inequality holds} 
\label{sec:Poincare}
\setcounter{equation}{0}
\setcounter{theo}{0}

In this section we restrict our analysis to {\bf Case 1},  that is when furthermore the steady state $G$ to the Fokker-Planck
equation  \eqref{eq:FPeq} is such that  the {\it weak weighted Poincaré-Wirtinger inequality}  \eqref{eq:wPineq} holds true and   the function $V := -\log G$ is of order $\langle x \rangle^\gamma$, that is it satisfies  \eqref{eq:Vsimxgamma}. We startb by showing the probably classical fact that \eqref{eq:wPineq} holds in the case when $c_1 = c_2$ in \eqref{eq:Vsimxgamma}. 
The proof uses some  of the arguments exposed in \cite{MR2386063,MR2609591}. 

\begin{lem}\label{lem:wwPineq:case1} Consider a potential function $V$ such that \eqref{eq:Vsimxgamma} holds with $c_1 = c_2$ and denote $G := e^{-V}$ (that we assume to have mass $1$).
There exists $\mu \in (0,\infty)$ such that the following weak weighted  Poincar\'e inequality 
$$
\int |\nabla h |^2 G \,  dx \ge \mu \int h^2 \langle x \rangle^{2(\gamma-1)} G \,  dx
$$
holds for any $h \in C^1_b(\R^d)$ such that $\M(hG) = 0$.  
\end{lem}

\begin{proof}[Proof of Lemma~\ref{lem:wwPineq:case1}.] We observe that $|\nabla V(x)|^2 \sim c_1^2 \gamma^2 |x|^{2\gamma-2}$ and $|\Delta V(x)| \sim c_1 \gamma |\gamma +d -2| |x|^{\gamma-2}$, so that 
\beqn \label{chCoercive:Cond1NVDV}
\forall  \, \theta \in (0,1), \,\, \exists R_1 > 0, \quad \theta |\nabla V|^2 \ge \Delta V \,\,\hbox{on}\,\, B_{R_1}^c.
\eeqn

We fix $h \in \DD(\R^d)$ such that $\M( h G ) = 0$. 
On the one hand, defining $g = h \, e^{-a V}$ and using the identity 
$$
\int |\nabla h|^2 \, e^{-V} 
=
\int |\nabla g|^2 \, e^{(2a-1)V} + a \int h^2 [ (1-a) |\nabla V|^2 -  \Delta V ] \, e^{-V}
$$
with $a := (1-\theta)/2$ together with \eqref{chCoercive:Cond1NVDV}, we obtain a first estimate
\beqn \label{eq;PoincareBRc}
{(1-\theta)^2\over 4}   \int h^2  | \nabla V |^2   \, e^{-V} 
\le 
\int |\nabla h|^2 \, e^{-V} +b \int_{B_{R_0}} h^2    \, e^{-V} .
\eeqn
On the other hand,   using the Poincar\'e-Wirtinger inequality   in $H^1(B_{R},Gdx)$, we get
\beqn\label{eq:PW-R}
\int_{B_{R}} h^2    \, {G \over \M(G{\bf 1}_{B_R})} \le C_R \, \int_{B_{R}} |\nabla h|^2 \,  {G \over \M(G{\bf 1}_{B_R}) } + \Bigl( \int_{B_{R}} h \,  {G \over \M(G{\bf 1}_{B_R}) } \Bigr)^2,
\eeqn
with possibly $C_R \to \infty$ when $R \to \infty$. 
Using the vanishing moment condition $\M( h G) = 0$ and the Cauchy-Schwarz inequality, we also have
\beqn\label{eq:PW-CS}
 \Bigl( \int_{B_{R}} h \, G \Bigr)^2 
= \Bigl( \int_{B^c_{R}} h \, G \Bigr)^2 \le
\eps_R \int_{B^c_R} h^2 \,  | \nabla V |^2 G,
\eeqn
with 
$$
\eps_R :=   \int_{B^c_R}    | \nabla V |^{-2}\, e^{-V} \to 0
\quad\hbox{as}\quad R\to\infty.
$$
Gathering \eqref{eq:PW-R} and \eqref{eq:PW-CS} for $R \ge R_2$, we obtain the second estimate
\bear\nonumber
 \int_{B_{R}} h^2    \, e^{-V} 
&\le& C_R \, \int_{B_{R}} |\nabla h|^2 \, e^{-V} \, dx + 2  \Bigl( \int_{B_{R}} h \, e^{-V} \, dx \Bigr)^2
\\ \label{eq:PW2eme}
 &\le& C_R \, \int_{B_{R}} |\nabla h|^2 \, e^{-V} \, dx + 2\eps_R \int_{B^c_R} h^2 \,  | \nabla V |^2\, e^{-V}.
\eear
Inserting \eqref{eq:PW2eme} into \eqref{eq;PoincareBRc} with $R \ge \max(R_1,R_2)$, we obtain
\bean
{(1-\theta)^2\over 4} \int h^2   | \nabla V |^2   \, e^{-V}
\le
(1 + C_R) \, \int |\nabla h|^2 \, e^{-V} 
 + 2\eps_R  \int_{B^c_R} h^2 \,  | \nabla V |^2\, e^{-V},
\eean
 and we conclude by taking  $R$ large enough in such a way that $2\eps_R < {(1-\theta)^2/ 4} $.  
\end{proof}

\subsection{Decay estimate in a small space}  \label{subsec:RocknerWang}

We first recall the following decay estimate in a small space. We define 
$$
E_1 := L^2 (G^{-1/2}), \quad E_2 := L^\infty(G^{-1}).
$$

\begin{theo}[R\"ockner-Wang \cite{MR1856277}] \label{theo:RocknerWang} Under assumptions \eqref{eq:wPineq}--\eqref{eq:Vsimxgamma} on the steady state $G$, set 
\begin{equation}\label{eq:Def-ThetaG}
\Theta_{G^{-1}}(t) := \exp(- \lambda t^{\gamma/(2-\gamma)}).
\end{equation}
Then there exists $\lambda^*$ such that for $\lambda < \lambda_{*}$ the semigroup $S_\LL$ satisfies for all $t > 0$
\beqn\label{eq:DecaySmallRevers}
\|S_\LL(t) f_0 - \M(f_0)\, G \|_{E_1} \lesssim \Theta_{G^{-1}}(t) \, \|  f_0 - \M(f_0)\, G \|_{E_2}. 
\eeqn
\end{theo}

We  briefly sketch a proof of Theorem~\ref{theo:RocknerWang} which uses as a cornerstone the  weak Poincar\'e inequality \eqref{eq:wPineq}. 
We refer to \cite{MR1856277} for a more detailed proof which uses closedly related (but different) arguments and more precisely which are based on the so-called {\it weak Poincar\'e inequality} (see \cite{MR1856277,MR2609591,MR2381160}). 

\begin{proof}[Proof of Theorem~\ref{theo:RocknerWang}]
We define $V := - \log G$, we set   ${\bf F}_0 := {\bf F} - \nabla V$ and we observe that ${\rm div}({G\bf F}_0) = 0$. The Fokker-Planck equation \eqref{eq:FPeq} can be rewritten as
\begin{equation}\label{eq:FP-GF0}
 \partial_t f = {\rm div} [  G  \nabla (fG^{-1}) + f {\bf F}_0 ]. 
\end{equation}
Now, since $f{\bf F}_{0} = fG^{-1}\cdot G{\bf F}_{0}$ and ${\rm div} (G{\bf F}_{0}) = 0$, we have 
$${\rm div}(f{\bf F}_{0}) = \nabla(fG^{-1})\cdot G{\bf F}_{0}.$$ 
Thus if $j : {\Bbb R} \longrightarrow {\Bbb R}$ is a locally Lispchitz function such that $j(fG^{-1}) \in L^1$, we see that $j'(fG^{-1}) {\rm div}(f {\bf F}_0) = \nabla j(fG^{-1})\cdot G{\bf F}_{0}$ and therefore
\begin{equation}\label{eq:fF0-j} 
\int\!\! j'(fG^{-1}) {\rm div}(f {\bf F}_0)  = \int \nabla j(fG^{-1})\cdot G{\bf F}_{0} = - \int\!\!  j(fG^{-1}) {\rm div}( G{\bf F}_{0}) = 0.
\end{equation}
In particular, taking $j(s) := s^2/2$ we see that
 $$
 \int fG^{-1} {\rm div}(f {\bf F}_0)  = 0, 
 $$
and as a consequence, multiplying \eqref{eq:FP-GF0} by $fG^{-1}$ we obtain
\bear\nonumber
{1 \over 2} {d \over dt}  \int f^2 G^{-1} 
& = &  - \int G\, |  \nabla (fG^{-1})|^2 + \int {\rm div} [ f {\bf F}_0 ] fG^{-1}
\\ \label{eq:small1}
&\le& - \mu \, \int f^2 \, \langle x \rangle^{2\gamma-2} \, G^{-1}, 
\eear
thanks to the weak Poincar\'e inequality \eqref{eq:wPineq}.  

More generally, for any convex function $j : \R \to \R$, locally in $W^{2,\infty}$, using \eqref{eq:fF0-j} on sees easily that the following generalized relative entropy inequality  holds
 (see \cite{MR2162224}, and also \cite{MR1856277} and the references therein)
 \bean
 {d \over dt}    \int j(f G^{-1}) \, G  
 & = &  \int j'(f G^{-1})  {\rm div} [  G \nabla (fG^{-1}) ]  + \int j'(f G^{-1}) {\rm div} [ f {\bf F}_0 ]  
\\
&=&  - \int j''(f G^{-1}) |\nabla (fG^{-1})|^2  \le 0.
\eean
In particular, taking $j(s) = |s|^p$, we obtain
$$\|f(t)G^{-1}\|_{L^p} \leq \|f_{0}G^{-1}\|_{L^p},$$
and passing to the limit $p\to\infty$, we get  
\beqn\label{eq:small2}
\| f(t)G^{-1}\|_{L^\infty} \le \|f_0 G^{-1}\|_{L^\infty},   \quad \forall \, t \ge 0.
\eeqn
We obtain \eqref{eq:DecaySmallRevers} by gathering the two estimates \eqref{eq:small1} and \eqref{eq:small2}. More precisely, 
we write, with $k := 2 - 2\gamma$,
\begin{align*}
\int f^2 G^{-1} 
&\le \int_{B_R} f^2 \, G^{-1} 
+ \int_{B_R^c} f^2 \, G^{-1}
\\
&\le R^{k} \int_{B_R} f^2 \, \, G^{-1} \, \langle x \rangle^{-k}
+ \| fG^{-1} \|_{L^\infty}^2  \int_{B_R^c}  G.
\end{align*}
Together with \eqref{eq:small1} and \eqref{eq:small2}, we get for $R \ge R_0$
\begin{equation*} 
{d \over dt}  \int f^2 G^{-1} 
\le - R^{-k} \int f^2 G^{-1}  +  \| f_0G^{-1} \|_{L^\infty}^2   \int_{B_R^c}  {\rm e}^{- c_1 |x|^\gamma}.
\end{equation*}
Integrating in time, for any $\theta \in (0,1)$, we find a constant $C_{\theta} > 0$ such that
\begin{align*}
\int f^2 G^{-1} 
  &\le  {\rm e}^{-t\, R^{-k}}    \int f^2_0 \, G^{-1} + C_\theta \, {\rm e}^{- \theta c_1  R^\gamma} \, \| f_0G^{-1} \|_{L^\infty}^2
 \\
 &\le {\rm e}^{- \lambda \, t^{\gamma/(2-\gamma)} }  \, \| f_0G^{-1} \|_{L^\infty}^2, 
\end{align*}
  by choosing $R := (t/(\theta c_1))^{1/(2-\gamma)}$ for $t$ large enough and defining $\lambda := (\theta c_1)^{- 1/(2-\gamma)}$,  
 which is nothing but \eqref{eq:DecaySmallRevers}. 
 \end{proof}

\subsection{Rate of decay in a large space}\label{subsec:rateLargeSpace}
We now present the proof of our main Theorem~\ref{th:MAIN} in the {\bf Case 1}  when furthermore $G$ satisfies the weak  Poincar\'e inequality 
\eqref{eq:wPineq} and the asymptotic estimates  \eqref{eq:Vsimxgamma}. 
We recall that  the projection operator $\Pi$ is defined by
$$
\Pi f := \M(f)\, G.
$$

\medskip\noindent
{\bf Step 1. } Here we consider the case $p \in [1,2]$. We begin by fixing $m$ a polynomial or exponential weight function, 
and we introduce a  stronger confinement exponential weight $m_0 (x):= \exp(\kappa_0 \, \langle x \rangle^\gamma)$, with $\kappa_0 \in (0,1/\gamma)$. 
We denote by $\Theta_m(t)$, $\Theta_{m_0}(t)$ and $\Theta_{G^{-1}}(t)$ the associated rate of decay  defined in \eqref{eq:MainEstimPoly}, \eqref{eq:MainEstimExpo1} and  in the statement of Theorem~\ref{theo:RocknerWang}, and we may assume that $\Theta_{m_0}/\Theta_m, \Theta_{G^{-1}}/\Theta_m \in L^1(0,\infty)$. 
We split the semigroup on invariant spaces 
$$
S_\LL = \Pi \, S_\LL + (I-\Pi) \, S_\LL , 
$$
and together with the iterated Duhamel formula \eqref{eq:Duhamel-gene} with $n_{1} = 0$ and $n_2 := n \ge d/2 +1$, 
we can write 
$$ S_\LL - \Pi = S_{\LL}(I- \Pi) =  (I - \Pi)S_{\LL}= T_{1} + T_{2},$$
where  
\bear
T_{1} & := & 
(I-\Pi) \, \left(S_\BB + \sum_{\ell=1}^{n-1} (S_\BB \AA)^{(*\ell)} *  S_\BB   \right) \label{eq:Def-T1} \\
T_{2} &:= & \left((I-\Pi) S_\LL \right)  \, * (\AA S_\BB)^{(*n)}. \nonumber
\eear
In order to estimate the first term $T_1$,  we recall that 
$$
\|  S_\BB(t) \|_{L^p(m) \to L^p} \le \Theta_{m}(t), \quad
\|  S_\BB(t) \AA \|_{L^p  \to L^p} \le \Theta_{m_0}(t) , 
$$
thanks to Lemma~\ref{lem:SBLpmLp}.  We write 
$$
\| T_1  \|_{\BBB(L^p(m),L^p)}  \le u_0+ ... + u_{n-1}
$$
where, for $\ell \in \{1, ..., n-1 \}$, we set
$$
 u_\ell := C  \, \|  S_\BB \AA \|_{\BBB(L^p)}^{(*\ell)} *   \|  S_\BB \|_{\BBB(L^p(m),L^p)}. 
$$
Since clearly $\Theta_m^{-1}(t) \le   \Theta_m^{-1}(s) \,  \Theta_m^{-1}(t-s)$ for all $0 \le s \le t$, we deduce 
$$
 \|  \Theta_m^{-1} u_\ell \|_{L^\infty_t} \le   \|  \Theta_m^{-1} \, S_\BB \AA \|_{L^1_t(\BBB(L^p))}^\ell \,    \|  \Theta_m^{-1} \, S_\BB \|_{L^\infty_t(\BBB(L^p(m),L^p))}, 
$$
and then 
$$
\forall \, t \ge 0, \qquad \|  T_1(t)  \|_{\BBB(L^p(m),L^p)} \le C \, \Theta_m(t) .
$$

\smallskip
In order to estimate the second term $T_2$, we recall that from Lemma~\ref{lem:SBLpmLp}, Lemma~\ref{lem:SBLpLq} and Theorem~\ref{theo:RocknerWang}, 
we have 
\bean
&& \|  \AA S_\BB(t) \|_{\BBB(L^p(m),L^1(m_0))} \le \Theta_{m}(t), \\
&& \|  (\AA S_\BB)^{(*(n-1))} (t) \|_{\BBB(L^{1}(m_0),L^{\infty} (m_1) )} \le  \Theta_{m_0}(t),  \\
&& \|  S_\LL(t) (I-\Pi) \|_{\BBB(L^\infty(G^{-1}),L^2(G^{-1/2}))} \le \Theta_{G^{-1}}(t),
\eean
where  $m_1 := m_0 + G^{-1}$. We thus obtain that $T_2$ satisfies the same estimate as $T_1$. 
We deduce \eqref{eq:MainEstim} by gathering the two estimates obtained on $T_1$ and $T_2$. 

\medskip\noindent
 {\bf Step 2. } We consider next the case  $p \in [2,\infty]$. Thanks to the  iterated Duhamel formula \eqref{eq:Duhamel-gene} with $n_{1} := n \geq d/2 + 1$ and $n_{2} = 1$, and using the fact that
$$(S_{\BB}\AA)^{(*n)} = S_{\BB}*(S_{\BB}\AA)^{(*(n-1))}*\AA,$$
we can write
$$S_\LL - \Pi = T_{1} + T_{3}$$ 
where $T_{1}$ is defined again by \eqref{eq:Def-T1} and 
$$T_{3} := (I - \Pi)S_\BB * (\AA S_\BB)^{(*(n-1))} * \AA S_{\LL} *(\AA S_\BB)$$
and we conclude in a similar way as in Step 1. \qed 


\section{The general case - the stationary problem}  
\label{sec:NonPoincareStatProblem}
\setcounter{equation}{0}
\setcounter{theo}{0}

In this section, we present the proof of the existence and uniqueness of  the solution to the stationary problem as claimed in Theorem~\ref{th:wellposedness} in the general case, that is assuming that the force field satisfies conditions \eqref{eq:CondChamp}--\eqref{eq:CondChamp2-b}, without any further assumption on the existence and particular properties of a positive stationary state.

In order to do so, we define the ``small'' Hilbert  space
$$
X :=  L^2(m_{0}), \quad m_{0} := \exp(\kappa \, \langle x \rangle^\gamma), \quad \kappa \in (0,1/ \gamma),  
$$
and we study some of the spectral properties of the opertaor $\LL$ acting in $X$. 
We consider below the eigenvalue problem on the imaginary axis. We start recalling some standard notions and notations. For an operator $\Lambda$ acting in a Banach space $X$, 
we denote $N(\Lambda) := \Lambda^{-1}(\{ 0 \})$ its null space, $\Sigma(\Lambda)$ its spectrum, $\Sigma_P(\Lambda)$ its point spectrum set, that is the set of the eigenvalues of $\Lambda$, while $\rho(\Lambda) := \C \backslash \Sigma(\Lambda)$ will denote its resolvent set and $R_\Lambda(z)$, for $z\in \rho(\Lambda)$, the resolvent operator defined by 
$$
R_{\Lambda}(z) := (\Lambda-z I)^{-1} .
$$
We denote $X_+$  the positive cone of $X$. 
In this section we shall prove 

\begin{theo}\label{th:SigmaL} There exists a unique steady state $G \in X$, positive, normalized with total mass $\M(G) = 1$, such that 
\beqn\label{eq:NLk=CG}
N(\LL^n) = {\rm span}(G), \quad \forall \, n \ge 1.
\eeqn
Moreover, there holds
\beqn\label{eq:SigmaPL=0}
\Sigma_P(\LL) \cap \left\{z \in {\Bbb C} \; ; \; {\Re e}\, z \geq 0 \right\} = \{ 0 \}.
\eeqn
Finally, estimate \eqref{eq:EstimGLinfty} holds true. 
\end{theo}

We start by recalling some elementary properties regarding the positivity of the semigroup $S_{\LL}(t)$.

\begin{prop}\label{prop:positiveSG} The operator $\LL$ and its semigroup $S_\LL$ satisfy the following  properties:\hfil\break
{\bf (a)}  The semigroup $S_\LL$ is positive, namely  $S_\LL(t) f \ge 0$ for any $f \ge 0$ and $t \ge 0$.\hfil\break
{\bf (b)}  The operators $\LL$ and $\LL^*$  satisfy Kato's inequality, that is if $f \in D(\LL)$ is complex valued, then we have 
\begin{equation}\label{eq:KatoIneq-C}
\LL |f| \geq {1 \over |f|} {\Re e}(\overline{f}\, \LL f),
\qquad 
\LL^* |g| \geq {1 \over |g|} {\Re e}(\overline{g}\, \LL^* g)
\qquad
\mbox{in }\, {\mathcal{D}}'({\Bbb R}^d).
\end{equation}
The above inequality means in particular that for all $\psi \in D(\LL^*)\cap X_{+}$ and all $\phi \in D(\LL)\cap X_{+}$, we have
$$
\langle |f|, \LL^* \psi \rangle  \ge {\Re e} \langle |f|^{-1} \,\overline{f}\, \LL f,\psi \rangle, \qquad
\langle \LL\phi, |g|\rangle  \ge {\Re e} \langle \phi,|g|^{-1}\, \overline{g} \, \LL g \rangle.
$$
{\bf (c)} The operator  $-\LL$ satisfies a ``weak maximum principle", namely  for any $a > 0$ and $g \in X_+$, there holds
\beqn\label{eq:PM*}
f \in D(\LL) \mbox{ and }  (- \LL + a) f = g \quad\mbox{imply}\quad f \ge 0. 
\eeqn
{\bf (d)} The opposite of the resolvent operator is a positive operator, namely  for any $a > 0$ and $g \in X_+$, there holds $- R_\LL(a) g \in X_+$. 

\end{prop} 

\begin{proof}[Proof of Proposition~\ref{prop:positiveSG}.]   
The argument to establish \eqref{eq:KatoIneq-C} is a classical one, but we outline it briefly.  For a smooth complex valued function $f$ if we set $f_{\epsilon} := (\epsilon^2 + |f|^2)^{1/2} - \epsilon$, then for $1 \leq j \leq d$ we have
\begin{align}
\partial_{j}f_{\epsilon} & = {{\Re e}(\overline{f}\, \partial_{j}f) \over (\epsilon^2 + |f|^2)^{1/2}}, \label{eq:j-epsilon-1}\\
\partial_{jj}f_{\epsilon} &= {{\Re e}(\overline{f}\, \partial_{jj}f) \over (\epsilon^2 + |f|^2)^{1/2}} +
{|\partial_{j}f|^2 \over (\epsilon^2 + |f|^2)^{1/2}} -
{\left({\Re e}(\overline{f}\, \partial_{j}f)\right)^2 \over (\epsilon^2 + |f|^2)^{3/2}}\, \label{eq:j-epsilon-2}.
\end{align}
Observe that $f_{\epsilon} \to |f|$ in $H^1_{\rm loc}({\Bbb R}^d)$ as $\epsilon \to 0$, we have $\partial_{j}|f| = |f|^{-1} {\Re e}(\overline{f}\,\partial_{j}f)$ in $H^1_{\rm loc}({\Bbb R}^d)$. Also  since 
$${|\partial_{j}f|^2 \over (\epsilon^2 + |f|^2)^{1/2}} -
{\left({\Re e}(\overline{f}\, \partial_{j}f)\right)^2 \over (\epsilon^2 + |f|^2)^{3/2}} \geq 0,$$
we conclude that $\partial_{jj}f_{\epsilon} \geq f_{\epsilon}^{-1}\,{\Re e}(\overline{f}\,\partial_{jj}f)$. Passing to the limit in the sense of the distributions we obtain $\partial_{jj}|f| \geq |f|^{-1} \, {\Re e}(\overline{f}\,\partial_{jj}f)$, and using the expressions of the operators $\LL$ and $\LL^*$ we obtain \eqref{eq:KatoIneq-C}.
\medskip

The properties (a), (c) and (d) are classical consequences of 
 \eqref{eq:KatoIneq-C}, applied to real valued functions, see for instance \cite{AGLMNS} or \cite{Mbook*}.
\end{proof} 

\begin{rem}\label{rem:module-f}
For later use, we point out that if $f \in H^2_{\rm loc}({\Bbb R}^d)$ is complex valued and moreover is such that $|f| > 0$ on ${\Bbb R}^d$, then as $\epsilon \to 0$ one may pass to the limit in \eqref{eq:j-epsilon-1} and \eqref{eq:j-epsilon-2} and obtain the following equalities
\begin{align}
\partial_{j}|f| & = {{\Re e}(\overline{f}\, \partial_{j}f) \over |f|}, \label{eq:j-1}\\
\partial_{jj}|f| &= {{\Re e}(\overline{f}\, \partial_{jj}f) \over |f|} +
{|\partial_{j}f|^2 \over |f|} -
{\left({\Re e}(\overline{f}\, \partial_{j}f)\right)^2 \over |f|}\, \label{eq:j-2}.
\end{align}
\end{rem}
   
\begin{prop}\label{prop:PMstrong} The operator $-\LL$ satisfies a ``strong maximum principle", namely for any given real valued function $f \in X \backslash \{0\}$, 
 there holds
 $$
f,|f| \in D(\LL), \quad \LL f = 0 \, \mbox{ and }\,  \LL    |f|  = 0
\quad\mbox{imply} \quad f > 0 \, \mbox{ or }\,  f < 0. 
$$
\end{prop}

\begin{proof}[Proof of Proposition~\ref{prop:PMstrong}.]  
Consider $f \in X \backslash \{ 0 \} $ such that $f  \in D(\LL)$  and $    \LL f  = 0 $. By a bootstrap regularization argument, we classically have $f \in C(\R^d)$. By assumption, there exist then $x_0 \in \R^d$, and two constants $c,r > 0$ such that $|f|  \ge c$ on  $B(x_0,r)$.  
From Lemma~\ref{lem:SBLpmLp}, we also have that the operator $\LL-aI$ is dissipative for $a$ large enough, in the sense that 
\beqn\label{eq:dissipL-a}
\forall \, f \in D(\LL) \quad ( ( \LL - a)f|f)_X \le - \|  f \|_X^2.
\eeqn
For instance one can take $a := M+1$, where  $M$ is the constant entering in the definition of $\BB$, using the fact that $\BB$ is then dissipative and that the same holds for $\LL - MI$ because $0 \le \AA \le M$. 

We next observe that for $\sigma > 0$ large enough, the function 
$$g(x) := c \exp(\sigma r^2/2 -\sigma |x-x_0|^2/2)$$
 satisfies $g=c$ on $\partial B(x_0,r)$ and 
$$
(- \LL +a) g = (a + d\sigma + \sigma {\bf F} \cdot (x-x_0) - {\rm div}{\bf F} - \sigma^2 |x-x_0|^2) \, g \le 0 \quad\mbox{on}\quad  B(x_0,r)^c. 
$$
 We define $h := (g- |f|)^+$ and $\Omega := \R^d \backslash {\overline B(x_{0},r)}$. We have $h \in H^1_0(\Omega,mdx)$ and 
\bean
 (\LL - a) h &\ge&  \theta'(g- |f|) \, \LL (g- |f|) - a \,  h
 \\
 &=& \theta'(g- |f|) \, [  (\LL - a) g + a |f|] \ge 0,
 \eean
 where we have used the notation $\theta(s) = s^+$. 
Thanks to a  straightforward generalization of \eqref{eq:dissipL-a} to $H^1_0(\Omega,m)$, we deduce
$$
0 \le ( (\LL - a) h|h)_{L^2(\Omega,m)} \le - \|  h \|^2_{L^2(\Omega,m)},
$$
and then $h=0$. This implies that we have $|f| \ge g $ on $\Omega = {\overline B(x_{0},r)}^c$, and thus $|f|  > 0$ on $\R^d$. Since $f  \in C(\R^d)$, we must have either  $f \geq g > 0$ or $f  \leq -g < 0$.  
\end{proof}

\medskip

\noindent
{\bf Proof of Theorem~\ref{th:SigmaL}. }  We split the proof into six steps. 

\medskip
\noindent{\bf Step 1. } We prove that there exists $G \in N(\LL)$ such that $G > 0$, which yields in particular that $0 \in \Sigma_P(\LL)$. In other words, we prove that there exists a positive and normalized (with mass $1$) steady state $G \in X$.  
For the equivalent norm $\Nt \cdot \Nt$ defined on $ X$ by 
$$
\Nt f \Nt := \sup_{t > 0}  \|  S_\LL(t) f \|_X, 
$$
we have $\Nt S_{\LL}(t)f \Nt \leq \Nt f \Nt$ for all $t \geq 0$, that is the semigroup $S_{\LL}$ is a contraction semigroup on $(X,\Nt\cdot\Nt)$.
There exists $R > 0$ large enough such that the intersection of the closed hyperplane $H:= \{f\; ; \; \M(f) = 1\}$ and the closed ball of radius $R$ in $(X,\Nt\cdot\Nt)$ is a convex, non empty subset. Then consider the closed, weakly compact convex set
$${\Bbb K} :=\left\{f \in X_{+} \; ; \; \Nt f\Nt \leq R, \quad \M(f) = 1\right\}.$$
Since $S_{\LL}(t)$ is a linear, weakly continuous, contraction in $(X,\Nt \cdot\Nt)$ and $\M(S_{\LL}(t)f) = \M(f)$ for all $t \geq 0$, we see that ${\Bbb K}$ is stable under the action of the semigroup. Therefore we apply the Markov--Kakutani fixed point theorem (see for instance \cite[Theorem 6, chapter V, \S~10.5, page 456]{Dunford-Scwhartz-1} or its (possibly nonlinear) variant \cite[Theorem 1.2]{EMRR}) and we conclude that there exists $G \in {\Bbb K}$ such that $S_{\LL}(t)G = G$. Therefore we have in particular $G \in D(\LL)$ and $\LL G = 0$. Moreover since $G \geq 0$, $G \not\equiv 0$ and  $G \in C({\Bbb R}^d)$, we may conclude that $G > 0$ on ${\Bbb R}^d$ by the strong maximum principle. 

\medskip\noindent{\bf Step 2. } In this step we prove that $N(\LL^2) = N(\LL)$, which implies that $N(\LL^n) = N(\LL)$ for all $ n \geq 1$. Otherwise, there would exist $g_{1}\in D(\LL) $ with $\|g_{1}\| = 1$ and $g_{2} \in D(\LL)$ such that $\LL g_{1} = 0$ and $\LL g_{2} = g_{1}$. Since $S_{\LL}(t)g_{1} = g_{1}$ for all $t \geq 0$, one sees easily that one must have $S_\LL(t) g_{2} = g_{2} + t g_{1}$. However, since $g_{1} \neq 0$, this is in contradiction with the fact that the semigroup $S_{\LL}$ is bounded on $X$ as stated in Lemma \ref{lem:LpmLL}. 

\medskip
\noindent{\bf Step 3. } We prove here that $N(\LL) = {\rm span}(G)$. 
Since $\LL$ is a real operator, we may restrict ourselves to real valued eigenfunctions. Consider a real valued eigenfunction $f \in N(\LL)$ with $f \not \equiv 0$. 
First we observe that thanks to Kato's inequality we have
\bean
0 =   ( \LL f  ) \, {\rm sgn} (f) 
\le \LL |f|. 
\eean
Actually this inequality must be an equality, since  otherwise we would have 
\bean
0 \not=   \langle \LL |f|,1  \rangle  = \langle  |f| , \LL^*1 \rangle  = 0, 
\eean
which is a contradiction. As a consequence, we have also $|f| \in N(\LL)$, so that the strong maximum principle, Proposition \ref{prop:PMstrong}, implies that either $f > 0$ or $f< 0$, and without loss of generality we may assume that $f > 0$, and up to a multiplication by a normalization factor, we may also assume that $\M(f) = 1$. Now, using once more Kato's inequality we have  
\bean
0 =  {\bf 1}_{[f-G > 0]} \LL (f-G) \,    \le  \LL (f-G)^+ , 
\eean
and due to the same reasons  as above, we may conclude that this last inequality is in fact an equality, that is $(f - G)^+ \in N(\LL)$. The strong maximum principle implies that either $(f-G)^+ \equiv 0$ or $(f-G)^+ > 0$ on ${\Bbb R}^d$. This means that either we have $f \le G$ or $f > G$ on ${\Bbb R}^d$. Thanks to  the normalization hypothesis $\M(f) = \M(G) = 1$, the second possibility must be excluded and thus we have $f \leq G$ on ${\Bbb R}^d$.
Repeating the same argument with $(G-f)^+$ we deduce that $G \le f$ and we conclude that  $f = G$.   

\medskip
\noindent{\bf Step 4. } We prove here that ${\rm i}{\Bbb R}\cap \Sigma_{P}(\LL) = \{0\}$:  that is that the only eigenvalue with vanishing real part is zero, or in other words, \eqref{eq:SigmaPL=0} holds. 
We consider a couple $(f,\mu)$ of eigenfunction and eigenvalue, with $\mu := {\rm i}\omega \in {\rm i}\R$, and normalized so that $\M(|f|) = 1$. Using the complex version of   Kato's inequality \eqref{eq:KatoIneq-C}, we have 
\beqn\label{eq:KatoLLasbolute}
\LL |f| \geq {1 \over |f|}{\Re e}(\overline{f}\, \LL f) = 0.
\eeqn
Computing $\langle \LL |f|, 1 \rangle$, thanks to the above inequality, we get  
$$
0 \le \langle \LL |f|, 1 \rangle = \langle  |f|, \LL^*1 \rangle = 0,
$$
which implies that the inequality is in fact an equality, that is $\LL |f| = 0$, and since $\M(G) = \M(|f|) = 1$, we conclude that $|f| = G > 0$. 

Next, using Remark \ref{rem:module-f} and \eqref{eq:j-1}--\eqref{eq:j-2}, 
 we have
\begin{align*}
|f|\,\LL |f| & = {\Re e}(\overline{f}\,\Delta f ) + 
{\Re e}(\overline{f}\, {\bf F}\cdot \nabla f) + 
{\rm div}({\bf F})|f|^2 + 
|\nabla f|^2 - {1 \over |f|^2}\, \left|{\Re e}(\overline{f}\, \nabla f)\right|^2\\
&= {\Re e}(\overline{f}\,\LL f ) +
|\nabla f|^2 - {1 \over |f|^2}\, \left|{\Re e}\left(\overline{f}\, \nabla f\right)\right|^2.
\end{align*}
Together with  $\LL  |f| = 0$ and ${\Re e}(\overline{f}\, \LL f) = 0$ (since $\LL f = {\rm i}\omega f$), the above identity implies
\begin{equation}\label{eq:f-multiple-G}
|\nabla f|^2 - {1 \over |f|^2}\, \left|{\Re e}(\overline{f}\, \nabla f)\right|^2 = 0.
\end{equation}
From this, as explained below, we infer that $f = \exp({\rm i}\,\theta)\, G$ for some constant $\theta \in [0,2\pi]$, and thus $\LL f = \exp({\rm i}\,\theta)\, \LL G = 0$ and $\omega = 0$. 

Indeed, in order to see that $f = \exp({\rm i}\,\theta)\, G$, for some $\theta\in [0,2\pi]$, let us write 
$$f = u + {\rm i}\,v$$
for two real valued functions $u$ and $v$. Then, since ${\Re e}(\overline{f}\, \nabla f) = u\nabla u + v\nabla v$, relation \eqref{eq:f-multiple-G} means that
$$(u^2 + v^2)\left(|\nabla u|^2 + |\nabla v|^2\right) = \left|u\nabla u + v\nabla v\right|^2,$$
which yields $v\nabla u - u\nabla v = 0$. Since $u$ and $v$ are both continuous functions and are not both identically equal to zero, we may assume for instance that there exists $x_0 \in \R^d$ such that $u(x_0) > 0$. Thus, if we denote by $\Omega$ the connected component of $\{ x \in \R^d; \, u(x) > 0 \}$ containing $x_0$, we have  $\nabla (v/u) = 0$ on $\Omega$. Hence  $v = \alpha\, u$ on $\Omega$ for some $\alpha \in {\Bbb R}$. However we must have $\Omega = \R^d$, since otherwise we would have $|f| = 0$ on $\partial\Omega$, which would be a contradiction with the fact that $G = |f| > 0$ in ${\Bbb R}^d$. Thus, setting $a := 1 + {\rm i}\alpha$ we get  $f = a\, u$ and thus $u = |u| = |f|/|a|$ is a positive steady state, so that $|u| = G/|a| > 0$ from Step~3. We conclue that $f = a\, G/|a| = \exp({\rm i}\,\theta) G$ for some $\theta \in [0,2\pi]$.

\medskip
\noindent{\bf Step 5. } Denoting $m := {\rm e}^{\kappa \langle x \rangle^\gamma}$, with $\kappa \in (0,1/\gamma)$, and $\alpha := \kappa \gamma (1- \kappa \gamma)/2 > 0$, the function $\psi^0_{m,1}$ 
being defined by \eqref{eq:Def-psi-mp} for $p=1$, $M,R > 0$ large enough so that 
$$
m \psi^0_{m,1} = \LL^* m   - M \chi_R m = \Delta m - {\bf F} \cdot \nabla m - M \chi_R m \le - \alpha m , 
$$
with $\alpha > 0$, we have 
$$
0 = \int (\LL G) \, m =  \int G \, (\LL^* m) = \int G \, (m \, \psi^0_{m,1} + M \chi_R m), 
$$
and  therefore
$$
 \int G \, m \le {1 \over \alpha} \int G \, M \chi_R m \le {M \over \alpha} \, \M(G). 
$$
We have proved $G \in L^1(m)$. 
Using the regularization property of $S_\BB$ established in Lemma~\ref{lem:Ultracon-2} and 
Lemma~\ref{lem:SBLpLq} as well as the same Duhamel formula as in Step~1 in the proof of Lemma~\ref{lem:LpmLL}, we deduce $G \in L^\infty(m)$.  

\medskip
\noindent{\bf Step 6. } We define $\kappa^*_1 := \gamma^{-1} \limsup_{|x| \to \infty} x \cdot {\bf F}(x) / |x|^\gamma$ and we know that $\kappa^* \in [1/\gamma,\infty)$ from assumptions \eqref{eq:CondChamp} and \eqref{eq:CondChamp2}. We then denote $g := c \, {\rm e}^{\kappa | x |^\gamma}$, with $\kappa > \kappa^*_1$, and we compute 
\bean
g^{-1} [(- \LL +a) g] &=& - \kappa^2 \gamma^2 |x|^{2(\gamma-1)} + \kappa \gamma {\bf F} \cdot x \, |x|^{\gamma-2} 
\\
&&+ a  - {\rm div}{\bf F} + \kappa\gamma (d+\gamma-2)|x|^{\gamma-2}  \le 0 \quad\mbox{on}\quad  B(0,R)^c, 
\eean
for some $R > 0$ large enough. Arguing as in the proof of Proposition~\ref{prop:PMstrong}, we deduce that $G \ge g$ on $B(0,R)^c$, and this yields the claimed lower bound (see also the proof of Lemma~\ref{lem:HarrisHyp1} below for similar aguments). \qed 

\section{The general case - decay estimate}  
\label{sec:NonPoincareSemiGroup}
\setcounter{equation}{0}
\setcounter{theo}{0}

In this section we present the proof of the our main Theorem~\ref{th:MAIN} in the general case, that is assuming that the force field satisfies conditions \eqref{eq:CondChamp}--\eqref{eq:CondChamp2-b}. We denote by $X_0$ the ``small" Banach space 
$$
X_0 := L^p(m_{0}), \quad m_{0} := \exp(\kappa \, \langle x \rangle^\gamma), \quad \kappa \in (0,\kappa^*_0), \quad p = 1 \mbox{ or } 2, 
$$
and we will establish a semigroup decay in that  space using three different strategies. 
The proof of the decay of the semigroup in a general Banach space $L^p(m)$, with a weight function $m$ which is a polynomial or an exponential function, then follows the same arguments as the ones which have been developed in section~\ref{subsec:rateLargeSpace} and
may be skipped here. 

\subsection{Modified Poincar\'e inequality approach} 
\label{subsec:modifPI}
In this section we fix the weight function $m=m_0$ as defined above.  We recall the so-called {\it Lyapunov condition} 
\beqn\label{eq:LyapunovPI}
\LL^*m_0 \le -  \xi_0+ M {\bf 1}_{B_{R_0}}, \qquad \xi_0 (x) :=  \zeta_0\, m_0(x) \langle x \rangle^{2(\gamma-1)},
\eeqn
which holds for some appropriate constants $\zeta_{0}, M, R_0 \in (0,\infty)$ and which has been established in Step~1 in the proof of Lemma~\ref{lem:SBLpmLp}.  We also recall that the stationary state $G$ given by Theorem~\eqref{th:SigmaL} satisfies estimate \eqref{eq:EstimGLinfty}, and thus a  local Poincar\'e (or Poincar\'e-Wirtinger) inequality on any ball with constants which may be estimated explicitely.  Under both above assumptions, we may adapt the proof of \cite[Theorem~3.6]{MR2381160} by Bakry, Cattiaux and Guillin in order to get the following {\it ``weak  Lyapunov-Poincar\'e  inequality''}  which is a generalisation of
the  {\it  weak weighted Poincaré-Wirtinger inequality}  \eqref{eq:wPineq} established in Lemma~\ref{lem:wwPineq:case1}. 
 
\begin{lem} 
\label{lem:BCGpoincare}  There exists a constant $\beta  \in (0,\infty)$ such that denoting $w_0 := 1 + \beta m_0$ and recalling that the function $\xi_{0}$ has been defined  by \eqref{eq:LyapunovPI}, the following inequality
\beqn\label{chMarkov:secHypo:WLPineq}
\frac14  \int h^2 \xi_0 G \le  \int \Bigl( w_0 |\nabla h|^2 - \frac12 h^2 \LL^* w_0 \Bigr)G, 
\eeqn
holds for all functions $h \in C^1_{b}({\Bbb R}^d)$ such that $\M( h G)= 0$.
\end{lem}

\begin{proof}[Proof of Lemma~\ref{lem:BCGpoincare}.]  
We take  $h \in C^1_{b}({\Bbb R}^d)$ such that $\M( h G)= 0$. From \eqref{eq:LyapunovPI}, we have
\bean
\int h^2  \xi_0 G   \le  \int h^2 \, \bigl( M \, {\bf 1}_{B_{R_0}} -  \LL^* m_0 \bigr) \, G.
\eean
For any $R \ge R_0$, the local Poincar\'e inequality writes 
\bean
\int h^2 \, {\bf 1}_{B_R} \, G 
&\le& \kappa_{R} \int_{B_R} |\nabla h|^2 \, G + {1 \over \M(G {\bf 1}_{B_R})}  \Bigl( \int_{B_R^c} h \, G \Bigr)^2
\\
&\le& {\kappa_{R} \over 1 + \beta}\int |\nabla h|^2 \, w_0 \, G +  {\M( \xi_0^{-1}G {\bf 1}_{B_R^c})  \over  \M(G {\bf 1}_{B_R})}  \int  h^2 \, \xi_0 G  ,
\eean
where we have used the the zero mass condition $\M( h G)= 0$ in the first line and the Cauchy-Schwarz inequality  in the second line. We then deduce 
\bean
\Bigl( 1 - M {\M( \xi_0^{-1}G {\bf 1}_{B_R^c})  \over  \M(G {\bf 1}_{B_R})} \Bigr) \int h^2  \xi_0 G   \le  {M \kappa_R \over 1 + \beta}\int |\nabla h|^2 \, w_0 \, G + \int h^2 \, ( -  \LL^* w_0 \bigr) \, G.
\eean
We may then first fix $R > R_0$  large enough and next $\beta > 0$   large enough in such a way that \eqref{chMarkov:secHypo:WLPineq} holds.  
\end{proof}

We introduce the weight function  $m := w^{1/2}_{0} G^{-1/2}$ and the associated Hilbert space $L^2(m)$.  
Despite the fact that in general the operator $\LL$ is not symmetric in $L^2(m)$, the associated
Dirichlet form has a nice positivity property. 
More precisely, in order to make the discussion simpler,  we shall consider the conjugate operator $h \mapsto Lh := G^{-1}\LL (Gh)$, which amounts to make the change of function $h := G^{-1}f$, and thus we define
$$
L h := G^{-1} \, \LL (G h) = 
 \Delta h + \Bigl( 2 {\nabla G \over G} + {\bf F} \Bigr) \cdot \nabla h. 
$$
For any smooth function $f$, defining the quadratic form $\EE(f) := (-\LL f|f)_{L^2(m)}$, we compute, upon integrating by parts 
\bean
\EE(f) 
&:=& ( -\LL f|f )_{L^2(m)} =  ( - L h|h )_{L^2( w_0^{1/2} G^{1/2})}
\\
&=& -  \int \left[ \Delta h + \Bigl( 2 {\nabla G \over G} + {\bf F} \Bigr) \cdot \nabla h   \right]   h w_0 G
\\
&=& \int |\nabla h|^2 G w_0 - \frac12\int \nabla (h^2)   \bigl[- G \nabla w_0 + w_0 \nabla G +   {\bf F}  w_0 G \bigr]
\\
&=&  \int |\nabla h|^2 G w_0  + \frac12\int  h^2   \bigl[-G \Delta w_0 + w_0 \Delta G +  {\rm div} ( w_0 G{\bf F}   ) \bigr]
\\
&=& \int |\nabla h|^2 G w_0  - \frac12\int  h^2 \,G \LL^*w_0 ,
\eean
where we have used the fact that $G$ is a steady state in the last line. 
For any smooth function $f$ such that $\M(f) = 0$, we deduce from the weak  Lyapunov-Poincar\'e inequality \eqref{chMarkov:secHypo:WLPineq} that 
\bean
\EE(f) 
\ge \frac14 \int f^2 \xi_0 G^{-1}. 
\eean
As a consequence, denoting $f(t) := S_\LL(t) f_0$ for $f_0 \in X_0$ with $\M(f_0) = 0$, we compute 
$$
{d \over dt} \int f^2 w_0 G^{-1} = - \EE(f)  \le - \frac14 {\zeta_0 \over 1+\beta}  \int f^2  \langle x \rangle^{2(\gamma-1)} w_0 G^{-1}. 
$$
This differential inequality is similar to \eqref{eq:small1} and we thus may conclude as in Section~\ref{sec:Poincare} that 
$$
\| f_t \|_{L^2} \lesssim \Theta_{G^{-1}}(t) \| f_0 \|_{L^2(m)}, \quad \forall \, t \ge 0. 
$$

\subsection{Harris-Meyn-Tweedie approach} 
\label{subsec:HarrisMeynTweedie}
We present now a second way to get a decay estimate in one small space. We start establishing a strict positivity estimate which will be the main 
technical estimate in our analysis.

\begin{lem}\label{lem:HarrisHyp1} For any $R > 0$, there exist $T > 0$ and  $\psi \in L^\infty(\R^d)$, $\psi \ge 0$, $\psi \not\equiv 0$,  such that 
$$
S_\LL(T) g_0 \ge \psi \int_{B_R} g_0 \, dx, \quad \forall \, g_0 \ge 0.
$$
\end{lem}

\begin{proof}[Proof of Lemma~\ref{lem:HarrisHyp1}.]  
We split the proof into three parts. 

\noindent
{\bf Step 1. }  
We prove that for any $p \in [1,\infty)$ and $m := \langle x \rangle^k$, $k > 0$ large enough, there exist $C, \alpha > 1$, such that for any $g_0 \in L^1(m)$, the 
function $g(t) := S_\LL(t) g_0$ satisfies
\beqn\label{eq:SBsmoothL1W1p}
\| \nabla  g(t) \|_{L^p} \le {C\over t^\alpha} \| g_0 \|_{L^1(m)}, \quad \forall \, t > 0.
\eeqn
First, as a consequence of Lemma~\ref{lem:SBLpLq} and the iterated Duhamel formula 
$$
S_\LL = S_\BB + \dots + (S_\BB \AA )^{(*(n-1))} * S_\BB + (S_\BB \AA )^{(*n)} * S_\LL, 
$$
we have that $S_{\LL}$ is ultra-contractive, that is
$$
\| g(t) \|_{L^\infty} \le {C\over t^{d/2}} \| g_0 \|_{L^1(m)}, \quad \forall \, t > 0.
$$
Next, we differentiate the Fokker-Planck equation and proceeding in a similar way as in the proof of Lemma~\ref{lem:SBL2H1}, we get
$$
\| g(t) \|_{H^1} \le {C\over t^{1/2}} \| g_0 \|_{L^2(m)}, \quad \forall \, t > 0.
$$
Repeating the proof of Lemma~\ref{lem:SBLpLq} for the derivative of $g(t)$, we may conclude in a quite standard way to \eqref{eq:SBsmoothL1W1p}.

\noindent
{\bf Step 2. }  
We prove that for any $t > 0, R > 0$, there exist $r,\lambda > 0$ and $x_0 \in B_R$ such that  
$$
S_\LL(t) g_0 \ge  \lambda \, {\bf 1}_{B(x_0,r)} \int_{B_R} g_0 \, dx. 
$$
We fix $R > 0$ and $g_0 \in L^1(\R^d)$ with ${\rm supp} \, g \subset B_R$. With the same notations as in the previous step, we have 
$$
\int g(t,x) \, dx = \int g_0(x) \, dx = \int_{B_R} g_0 \, dx.
$$
Moreover, thanks to Lemma~\ref{lem:LpmLL}, there exists $A  > 0$, such that 
$$
\int g(t,x) \, |x| \, dx \le A \int g_0 \, |x| \, dx \le AR \int_{B_R} g_0 \, dx.
$$
For any $\rho > 0$, , we write 
\bean
\int_{B_\rho} g(t,x)dx 
&=& \int g(t,x)dx - \int_{B_\rho^c} g(t,x) dx
\\
&\ge& \int g_0 dx - {1 \over \rho} \int g(t,x)  |x|dx
\\
&\ge & \Bigl( 1 - {AR \over \rho} \Bigr) \int_{B_R}  g_0 (x)dx  \ge \frac12  \int_{B_R}  g_0(x)dx, 
\eean
by choosing $\rho > 0$ large enough.  As a consequence we infer that there exists $x_0(t) \in B_\rho$ such that 
$$
g(t,x_0(t)) \ge  2\lambda := {1 \over 2|B_\rho|} \int_{B_\rho} g(t,x)dx \, dx \ge \frac1{4 |B_\rho|}  \int_{B_R}  g_0(x)dx.
$$
Thanks to \eqref{eq:SBsmoothL1W1p} and Morrey embedding $W^{1,p}({\Bbb R}^d) \subset C^{0,1/2}({\Bbb R}^d)$, which holds true for $p \in (1,\infty)$ large enough,
we know that there exists a constant $c_{0} > 0$ such that for all $t > 0$ and $x,y\in {\Bbb R}^d$
$$|g(t,x) - g(t,y)| \leq c_{0}\|\nabla g(t,\cdot)\|_{L^p}\cdot |x-y|^{1/2} \leq c_{0}C\|g_{0}\|_{L^1}\, t^{-\alpha}|x-y|^{1/2}.$$
Therefore, we have 
$$
g(t,x) \ge g(t,x_0(t)) - c_{0}C\|g_{0}\|_{L^1} t^{-\alpha} |x-x_0(t)|^{1/2} \ge \lambda
$$
for $x \in B(x_0(t),r)$, where $r$ is chosen by setting $c_{0}C\|g_{0}\|_{L^1}t^{-\alpha} r^{1/2} = \lambda$. 

\medskip\noindent
{\bf Step 3. } (Spreading of the positivity.)  We prove that that for any $r_0,r_1 >0$ and $x_0 \in \R^d$, there exist $t_1,\kappa_1 > 0$ (computable) such that  
$$
g_0 \ge {\bf 1}_{B(x_0,r_0)}  \quad\Rightarrow\quad g(t_1,\cdot) \ge \kappa_1  {\bf 1}_{B(x_0,r_1)}. 
$$
We first observe that $h(t,x) := g(t,x) \, {\rm e}^{t \| {\rm div} {\bf F} \|_\infty}$ is a super solution of the Fokker-Planck equation with initial data $g_{0}$, that is $\partial_{t}h - \LL h \geq 0$ and
$$ \partial_t h  \ge \LL^\sharp h := \Delta h + {\bf F} \cdot \nabla h. $$
We introduce the function 
$$
\varphi (t,x) := \varphi_0 - \eps, \quad \varphi_0 := (t+t_0)^{-\alpha} {\rm e}^{-\mu |x-x_0|^2/(t+t_0)}, 
$$
for some parameters $\alpha,\mu,\eps, t_0 > 0$ to be specified latter. We compute 
\bean
&&\partial_t \varphi = \bigl( \mu {|x-x_0|^2 \over (t+t_0)^2} -  {\alpha \over  t+t_0} \bigr) \varphi_0, 
\quad\nabla_x \varphi = - 2 \mu {x-x_0 \over t+t_0} \varphi_0, 
\\
&&
\Delta_x \varphi = \bigl(  4 \mu^2 {|x-x_0|^2 \over (t+t_0)^2}  - 2  {\mu d\over t+t_0} \bigr)\varphi_0. 
\eean
To simplify notations, we only consider the case $x_0 = 0$ (the general case
can be dealt exactly in the same way) and we denote $r = |x|$ and $\tau = t + t_0$, and we consider $\phi$ as function of $\tau$ and $r$. 
For $\mu \ge 1$, $T > 0$ and $t \in (0,T)$, we deduce
\bean
{\partial_t \varphi -\LL^\sharp \varphi\over \varphi_0} 
&=&  \mu {r^2 \over \tau^2} -   {\alpha \over \tau} +  2d \,  {\mu \over \tau} - 4 \mu^2 {r^2 \over \tau^2} + 2 \mu {{\bf F} \cdot x\over\tau}  
\\
&\le&  -   {\alpha \over \tau} - 2  \mu^2 {r^2 \over \tau^2}  +  2d \,  {\mu \over \tau}   + 4 \| {\bf F} \|_{L^\infty}, 
\eean
and then 
\beqn\label{eq:dtLsharp}
\partial_t \varphi \le \LL^\sharp \varphi  \quad\mbox{on}\quad (0, T) \times \R^d  
\eeqn
for any $\alpha \ge 2d\mu + 4 T  \| {\bf F} \|_{L^\infty}$. We fix $\eps$ such that 
$$
\varphi(0,r_0) =  t_0^{-\alpha} {\rm e}^{-\mu r_0^2/ t_0} - \eps = 0.
$$
We have built $\varphi$ such that 
$$
\partial_t (h-\varphi)  \ge \LL^\sharp (h-\varphi) \,\,\,  \mbox{on}\,\,\,  (0, T) \times \R^d , 
\quad  (h-\varphi)  (0) \ge 0 \,\,\,  \mbox{on}\,\,\,  \R^d .
$$
Because $(h- \varphi)_- \in L^1((0,T) \times \R^d)$, by the weak maximum principle we deduce that $h \geq \phi$ on $(0,T)\times {\Bbb R}^d$, so that
$$
g(t,x) \ge {\rm e}^{-t \| {\rm div} {\bf F} \|_{L^\infty}} \, \varphi(t,x) \,\,\,  \mbox{on}\,\,\,  (0, T) \times \R^d .
$$
We now define $r_1 := r_0 t_0^{-\beta} \tau^\beta $ as well as  
$$
\psi(t) := \varphi(t,r_1) =  \tau^{-\alpha} {\rm e}^{-\mu r_0^2  t_0^{-2\beta}\tau^{2\beta-1}} - t_0^{-\alpha} {\rm e}^{-\mu r_0^2/ t_0}.
$$
We have   (recall that $\tau = t + t_{0}$)
\beqn\label{eq:psi0}
\psi(0) = 0, \qquad \psi'(t) = \bigl( {\mu r_0^2 t_0^{-2\beta} (1-2\beta) \over \tau^{1-2\beta}} - \alpha \bigr) \tau^{-\alpha-1} {\rm e}^{-\mu r_0^2 \tau^{2\beta-1}} \ge 0
\eeqn
whenever $\alpha \le {\mu r_0^2 t_0^{-2\beta}(1-2\beta) \over \tau^{1-2\beta}}$. We choose $\mu=1+r_0^{-2}$, $T=2t_0 \le 1$ to be fixed below, $\beta = 1/4$
and $\alpha := 2d  + 4   \| {\bf F} \|_{L^\infty} \ge 2$. We observe that \eqref{eq:dtLsharp} is true on the time interval $(0,T)$ and \eqref{eq:psi0} is true on $(0,2t_0)$ by 
choosing $t_0 := 1/(4 \alpha) \le 1/2 $. In particular, $\psi(2t_0) \ge 0$. As a consequence, 
$$
\varphi(2t_0,x) \ge \bar\varphi_1 := \varphi(2t_0,r_{t_0} ) =  (3t_0)^{-\alpha} {\rm e}^{-\mu r_0^2 2^{-1/2} / t_0  } - t_0^{-\alpha} {\rm e}^{-\mu r_0^2/ t_0} > 0,
$$
for any $|x| < r_{t_0} := r_0 2^{1/4}$ and then 
$$
g(2t_0,x) \ge  \bar\varphi_1 {\rm e}^{-2t_0 \| {\rm div} {\bf F} \|_{L^\infty}} \, {\bf 1}_{B(x_0,r_0 2^{1/4})}, 
$$
by the above weak maximum principle. Iterating this argument, we get the announced estimate with 
$$
t_1 := 2 n t_0, \ n:= \bigl( 1 + \bigl( 4 {\log r_1/r_0 \over \log 2} \bigr] \bigr), 
\ \kappa_1 := {\rm e}^{-2t_1 \| {\rm div} {\bf F} \|_{L^\infty}} \Pi_{k=1}^n \bar\varphi_k  
$$
and 
$$
\bar\varphi_k  :=  (3t_0)^{-\alpha} {\rm e}^{-\mu \rho_k^2 2^{-1/2} / t_0  } - t_0^{-\alpha} {\rm e}^{-\mu \rho_k^2/ t_0}, \quad \rho_k := r_0 2^{k/4}.
$$

\medskip\noindent
{\bf Step 4 and conclusion. } From Step 2, we know that we may find $\tau_0, r_0,R,\lambda > 0$ and $x_0 \in B(0,R)$ such that  
$$
g(\tau_0, \cdot) \ge  c_0 \, {\bf 1}_{B(x_0,r_0)}, \quad c_0 := \lambda   \int_{B_R} g_0 \, dx. 
$$
Using Step 3 with $r_1 := 2R$, we obtain 
$$
g(\tau_0 + t_1, \cdot) \ge  \lambda \kappa_1 \, {\bf 1}_{B(0,R)} \int_{B_R} g_0 \, dx ,
$$
%
which is nothing but the announced Harris condition. 
\end{proof}

\medskip
We recall the following so called {\it ``abstract  subgeometric Harris-Meyn-Tweedie theorem''}. 

\begin{theo}[subgeometric Harris-Meyn-Tweedie] \label{th:DFGharris} 
Consider a semigroup $S$ in $L^1(m_0)$ which is positive and mass conservative. 
Assume furthermore that the following two additional conditions hold:\hfil\break
(1) the Lyapunov condition 
$$
\LL^* m \le - \varphi(m) + b, 
$$
for some weight function $m : \R^d \to [1,\infty)$ converging to infinity and  some positive, concave and converging to infinity function $\varphi : \R+ \to \R_+$ and some constant $b > 0$;  (2) the strict positivity (or irreducibility)  condition
$$
S_T f \ge \nu \int_{B_R} f, \quad \forall \, f \ge 0, 
$$
for any $R > 0$ and some $T > 0$ and $\nu \ge 0$, $\nu \not\equiv 0$. 
Then, there exists some constant $C > 0$ such that 
\beqn\label{eq:AbstractHarris}
\| S_t f_0   \|_{L^1} \le  {C \over (\varphi \circ H_\varphi^{-1})(t)} \| f _0   \|_{L^1(m)},
\eeqn
for any $f_0 \in L^1(m)$, $\M(f_0) = 0$, where $H_\varphi$ is the function defined by 
$$
H_\varphi(u) := \int_1^u{ds \over \varphi(s)}. 
$$
\end{theo}

Theorem~\ref{th:DFGharris} is stated in \cite[Theorem~1.2]{MR2381160} and it is proved in \cite[(3.5) in Theorem~3.2]{MR2499863} (see also \cite[Section~5.1]{MR2499863}) by the mean of probabilistic tools. A simple semigroup proof is presented in \cite{CanizoMischler}. A variant of  \cite[Theorem~1.2]{MR2381160} is also presented in \cite[Section~4]{Hairer}. 
Gathering the Lyapunov condition \eqref{eq:LyapunovPI}, the strict positivity property established in Lemma~\ref{lem:HarrisHyp1} and the subgeometric Harris-Meyn-Tweedie Theorem~\ref{th:DFGharris}, we deduce the following decay estimate.

\begin{cor}\label{cor:WeakHarris} For any $f_0 \in L^1(m_0)$ such that $\M(f_0) = 0$, 
the associated solution $  f(t, \cdot) = S_\LL(t) f_0$  to the Fokker-Planck equation \eqref{eq:FPeq} satisfies
\beqn\label{eq:EstimL1Harris}
\|f(t,.)  \|_{L^1 } \lesssim \Theta_{m_0}(t) \, \|f_0   \|_{L^1(m_0)},
\eeqn
which is nothing but the conclusion of Theorem~\ref{th:MAIN} in $L^1(m_0)$. 
\end{cor}

\noindent
{\sl Proof of Corollary~\ref{cor:WeakHarris}. }
We write \eqref{eq:CondChamp3} (or \eqref{eq:LyapunovPI}) as 
\beqn\label{eq:LyapunovHarris}
\LL^*m_0 \le - \varphi(m_0)  +  M,
\eeqn
with 
$$
\varphi(y) := C \, { y \over (\log y)^\alpha}, \quad \alpha := 2 {1-\gamma \over \gamma}.
$$
We easily compute 
$$
H_\varphi(u) = C \int_1^u (\log y)^\alpha {dy \over y} = C \int_1^{\log u} z^\alpha dz  = C (\log u)^{\alpha+1},
$$
so that $H^{-1}_\varphi(v) = \exp (c v^{{1 \over \alpha+1}})$. From \eqref{eq:AbstractHarris}, we conclude \eqref{eq:EstimL1Harris} with $\Theta_{m_0}(t) := K \exp(-\lambda t^{{\gamma \over 2-\gamma}})$, $K \ge 1$, $\lambda > 0$.  
 \qed

\subsection{Krein-Rutman type approach} 
\label{subsec:KRapproach}
We finally present a third approach which is mainly an adaptation of an argument used in the proof of  \cite[Theorem 2.1]{MS*}.
The key argument is an accurate estimate on the confinement process of the semigroup and it does not deeply use the positivity estimates. The drawback comes  from the fact that the rate of decay in the small space $X_0$ is lower than in the small space $L^2(G^{-1/2})$ (see section~\ref{subsec:rateLargeSpace} for the notations). 

 We consider the  sequence of spaces $(X_k)_{k\in \N}$, as defined in \eqref{eq:defXk}, and $X_\infty := L^2(m_0^{1/2})$. 
 We observe that  $X_k \subset X_{k+1} \subset X_\infty$ for any $k \in \N$. For $0 \leq \eta \leq 1$ we also denote 
 $X_{k,\eta}$ the space defined by Hilbertian interpolation between $X_{k,0} = X_k$ and $X_{k,1} := \{  f \in X_k; \,\, \LL f \in X_k \}$, that is, with the notations of L. Tartar \cite[Chapter 22, page 109]{Tartar-Interpol},
$$X_{k,\eta} := \left(X_{k},X_{k,1}\right)_{\eta,2}.$$
    
\begin{lem}\label{lem:ARBk}  Let us fix an integer $j >  2(1-\gamma)/\gamma > 0$. There exists a constant $C$ such that for any $\ell_1,\ell_2, k \in \N$, $k \ge  j $, $\ell_i \ge 1$, 
we have for all $z \in {\Bbb C}$ with $ {\Re e}\, z >0$  
\begin{align} 
\label{lem:RBXk}
& \|R_\BB(z)   \|_{ X_{k-j} \to X_k } \le  \CC_k := C \, k^{j}, 
\\
& \|\AA R_\BB(z)^{\ell_1}\AA R_\BB(z)^{\ell_2}  \|_{ X_0\to X_0 } \le (\ell_1! \ell_2!)^j C^{\ell_1+\ell_2}/ \langle y \rangle^{1/2}, \label{lem:ARBk-2} 
\end{align}
where $z = x + {\rm i}y$, $x,y \in \R$. 
\end{lem}

\begin{proof}[Proof of Lemma~\ref{lem:ARBk}.]  
We use the representation formula  
$$
R_\BB(z) = - \int_0^\infty {\rm e}^{-zt} \, S_\BB(t) \, dt
$$
together with the estimates established in Lemmas~\ref{lem:SBLpmLp}, \ref{lem:SBLpLq} and \ref{lem:SBL2weight} in the following way.
In order to simplify the presentation, we only consider in the sequel the boundary case $\Re e \, z = x= 0$. 

\smallskip
On the one hand, we define, as in Lemma~\ref{lem:SBL2weight}, $\alpha^* := 1/2(1-\gamma)$,  
so that $\alpha^* \gamma j > 1$, 
and we observe that the LHS term of \eqref{eq:SBXk-jXk} belongs to $L^1(\R_+)$.  Therefore, with the notations of Lemma~\ref{lem:SBL2weight},
 we have for any $y \in \R$
\begin{equation*}
\|R_\BB({\rm i}y) \|_{X_{k-j} \to X_k} 
\le C_1   + \Bigl( {k \over \kappa} \Bigr)^j \, {1 \over \alpha \gamma j -1} \le  C \, k^j . 
\end{equation*}

On the other hand, from \eqref{eq:SBL2H1}, we have  
\beqn\label{eq:RBXinftyH1}
\sup_{y \in \R} \| R_\BB  ({\rm i}y) \|_{X_\infty \to H^1} \le   \int_0^\infty \| S_\BB(t) \|_{X_\infty \to H^1}  \, dt  \le C.
\eeqn
The latter estimate together with \eqref{lem:RBXk} yield that,  for any $y \in \R$ and $\ell_{2} \in \N^*$, we have
\begin{align}
\|\AA R_\BB({\rm i}y)^{\ell_2}  \|_{ X_0\to X_{0,1/2} } &\leq 
\nonumber\\ 
&\hskip-3cm\le   
\|\AA R_\BB({\rm i}y) \|_{ X_\infty \to X_{0,1/2} } \, \| R_\BB ({\rm i}y) \|_{X_{(\ell_2 - 2) j} \to X_{(\ell_2-1) j} }  \,  \cdots \, \| R_\BB ({\rm i}y) \|_{X_0 \to X_j} 
\nonumber\\ 
&\le ( \ell_2!)^j C^{ \ell_2}.\label{eq:Estim-ell2}
\end{align}
On the other hand, from the identity 
$$
R_\BB(z) = z^{-1} (\RR_\BB(z)  \BB - I)
$$
and an interpolation argument, we deduce that 
$$
\|R_\BB({\rm i}y) \|_{X_{0,1/2} \to X_j}  \le C/ \langle y \rangle^{1/2},
$$
and therefore for any $\ell_{1}\in {\Bbb N}$
$$
\|  \AA R_\BB({\rm i}y)^{\ell_1}  \|_{ X_{0,1/2} \to X_0} \le ( \ell_1!)^j C^{ \ell_1}/ \langle y \rangle^{1/2}.
$$
It is now clear that the above estimate together with \eqref{eq:Estim-ell2} completes the proof of estimate \eqref{lem:ARBk-2}.
\end{proof}

\begin{lem}\label{lem:RL} Let us fix again an integer $j > 2(1-\gamma)/\gamma > 0$.   There exists a constant $C$ such that  for any   $\ell \in \N^*$, denoting by $\Pi$ the projection on $G$, that is $\Pi(f) := \M(f)G$, 
we have
\begin{equation}\label{eq:Estim-RL}
\sup_{  {\Re e}\, z > 0} \|  (I-\Pi) R_\LL(z) ^{ \ell} \|_{X_0 \to X_\infty } \le C^\ell \, (\ell !)^{j}.
\end{equation}
\end{lem}

\begin{proof}[Proof of Lemma~\ref{lem:RL}.]  
Since the operator $\LL-aI$ is dissipative for any $a > 0$, we clearly have 
\bean
C_{\LL,a} := \sup_{  {\Re e}\, z \geq a}  \|  (I-\Pi) R _\LL(z)  \|_{X_0 \to X_\infty } \lesssim 
\sup_{  {\Re e}\, z \geq a}  \|R _\LL(z)  \|_{X_0 \to X_\infty } < \infty, 
\eean
and thus we have only to prove that the constant $C_{\LL,a}$ does not blow up when $a \to 0^+$. 

\smallskip
\noindent
{\bf Step 1. }  We claim that for any fixed $M$, there holds 
\begin{equation}\label{eq:Estim-Mkj}
\sup_{z \in \Omega_M} \|  (I-\Pi) R_\LL(z)  \|_{X_{k-j} \to X_k }  \le C_M \, \CC_k, 
\end{equation}
where we define $ \Omega_M := \{  z = x + {\rm i}y \in \C, \,\, 0 < x \le 1, \, |y|  \le M \}$ and we recall that $\CC_k$ 
is defined in \eqref{lem:RBXk}. We argue by contradiction, assuming that 
there exists $y \in [-M,M]$ and a sequence $(z_n)$ in $\Omega_M$ such that 
$$
z_n \to z:= {\rm i}y \quad\mbox{and}\quad \CC_n^{-1} \|  R_{\LL_1}  (z_n) \|_{\BBB(X_{n-j},X_n)} \to \infty,
$$
with $\LL_1 := \Pi^\perp \LL$, where for brevity we note $\Pi^\perp := I - \Pi$, despite the fact that $\Pi$ is not an orthogonal projection. 
The last family of  blowing up estimates means that there exist sequences $(\tilde f_n)_{n}$ and $(\tilde g_n)_{n}$ such that 
$$
\CC^{-1}_n \|  \tilde f_n \|_{X_n} \to \infty,  \quad  \|  \tilde g_n \|_{X_{n-j}} = 1, \quad   \tilde f_n = R_{\LL_1}  (z_n) \, \tilde g_n,
$$
or, equivalently, that there exist $(f_n)_{n}$ in $X_n$ and $(g_n)_{n}$ in $X_{n-j}$ satisfying 
$$
 \| f_n \|_{X_n} = 1,  \quad \CC_n \|  g_n \|_{X_{n-j}} \to 0 , \quad g_n = (\LL_1 - z_n) \, f_n .
 $$
This in turn would imply that
\beqn\label{eq:RL1}
R_\BB(z_n) \AA \Pi^\perp f_n + \Pi^\perp f_n - z_n R_\BB(z_n) \Pi f_n =  R_\BB(z_n) g_n  ,
\eeqn
with 
$$
\|   R_\BB(z_n) g_n \|_{X_n}  \le  \CC_n \|  g_n \|_{X_{n-j}}   \to 0,
$$
by using \eqref{lem:RBXk}. Since $(f_n)_{n}$ is bounded in $X_n \subset X_\infty = L^2({\rm e}^{\kappa \, \langle x \rangle^\gamma})$, by weak compactness of this sequence in $X_{\infty}$, we find $f \in X_\infty$ and a subsequence denoted again by $(f_n)_{n}$ such that $f_n \wto f$ weakly in $X_\infty$, and then 
$\AA \Pi^\perp f_n \wto  \AA \Pi^\perp $ weakly in $X_0$. 
Together with \eqref{eq:RBXinftyH1}, we deduce that 
\beqn\label{eq:RL2}
R_\BB(z_n) \AA \Pi^\perp f_n \to R_\BB(z ) \AA \Pi^\perp f  \quad\mbox{ strongly in } \, X_0.
\eeqn
Now, passing (weakly) to the limit in \eqref{eq:RL1}, we have 
$$
 R_\BB(z) \AA \Pi^\perp f + \Pi^\perp f - z R_\BB(z) \Pi f = 0.
$$
We claim that $f \not= 0$. If not, we get from \eqref{eq:RL2} that  
$$ R_\BB(z_n) \AA \Pi^\perp f_n \to 0\qquad
\mbox{and}\qquad
\Pi f_n  \to  0\quad\mbox{strongly in }\,X_0$$
 and then together with \eqref{eq:RL1} that $\| \Pi^\perp f_n \|_{X_n} \to 0$. Thus we would have 
 $$ 1 = \|f_n\|_{X_n}  \le \|\Pi f_n \|_{X_0} + \|\Pi^\perp f_n \|_{X_n}  \to 0,$$
 which is a contradiction.
 As a consequence, we have exhibited an $f \in X_\infty \backslash \{  0 \}$ such that  $(\LL_1 - zI) f = 0$. This means that $f$ is an eigenvector for $\LL_1 = \Pi^\perp\LL$  associated to an eigenvalue $z \in {\rm i}\R$; however this is in contradiction with the fact that $\Sigma_{P}(\Pi^\perp\LL) \cap {\rm i} \R = \emptyset$. Thus the proof of \eqref{eq:Estim-Mkj} is complete. 

\smallskip
\noindent{\bf Step 2. }  In this step we complete the proof of the Lemma. We begin by recalling that $\LL = \AA + \BB$ and then we write 
$$
R_\LL(z) = R_\BB(z) - R_\LL(z) \AA R_\BB(z)
$$
and 
$$
R_\LL(z) (1 - \VV(z)) = R_\BB(z) - R_\BB(z) \, \AA \, R_\BB(z), \qquad\mbox{where }\,
\VV(z) := \left(\AA \, R_\BB(z)\right)^2.
$$
Let $\tilde X_k$ be defined as $X_k$ but with a coefficient $\tilde \kappa > \kappa$ so that $\tilde X_k \subset \tilde X_\infty \subset X_0 \subset X_k$ with embedding constants uniformly bounded with respect to $k$.
First we may fix $M$ large enough such that $\|\VV(z) \|_{\BBB(X_0)} \le 1/2$ and $\|\VV(z) \|_{\BBB({\tilde X}_0)} \le 1/2$, for any $z = x + {\rm i}y$, with $|y| \ge M$: this is indeed possible thanks to \eqref{lem:ARBk-2} by choosing $\ell_1 = \ell_2 = 1$. Next, we write the  expansion
$$
R_\LL (z) = R_\BB (z)-  \bigl( R_\BB (z)- R_\BB (z)\, \AA \, R_\BB (z)\bigr) \Bigl( \sum_{j=0}^\infty \VV(z)^j \Bigr) \AA R_\BB(z), 
$$
where the series converges normally in $\BBB(X_0)$ and in $\BBB({\tilde X}_{0})$. More precisely, for $z = x + {\rm i}y$ and $|y| \geq M$ we have 
\begin{align}
\| R_\LL (z) \|_{X_{k-j} \to X_k} &\le \| R_\BB (z) \|_{X_{k-j} \to X_k}  
\nonumber\\
&\hskip-1cm +  \| (R_\BB - R_\BB \, \AA \, R_\BB)(z)\|_{\tilde X_0 \to X_0}  \Bigl( \sum_{j=0}^\infty \| \VV(z) \|_{\BBB(\tilde X_0)}^j \Bigr) \, \| \AA R_\BB \|_{X_{k-j} \to \tilde X_0} , \nonumber\\
&\hskip-1cm\leq \| R_\BB (z) \|_{X_{k-j} \to X_k} + 2 \, 
\|(R_\BB - R_\BB \, \AA \, R_\BB)(z)\|_{\tilde X_0 \to X_0} \,
\|\AA R_\BB \|_{X_{k-j} \to \tilde X_0}. \label{eq:estim-Serie}
\end{align}
The right hand side of the above inequality being bounded by a constant $C \, \CC_k$, we conclude that  
$$
\sup_{| {\Re e}\, z| \geq M} \|(I-\Pi) R_\LL(z)\|_{X_{k-j} \to X_k }  \le C  \, \CC_k. 
$$
Together with \eqref{eq:Estim-Mkj}, we conclude the proof of \eqref{eq:Estim-RL} in the case when $\ell=1$. For the general case $\ell \ge 1$, we 
argue similarly as we did in the proof of \eqref{lem:ARBk-2}. 
\end{proof}

\begin{theo}\label{th:SLsmall} Let $\sigma^*_\LL := 1/\lfloor 2/\gamma\rfloor$, where $\lfloor s \rfloor$ stands for the integer part of the real number $s$, 
and let $m_{0}(x) := \exp(\kappa \xx^\gamma)$ and $0 < \kappa\gamma < 1/8$. Then for any $\sigma \in (0,\sigma^*_\LL]$ and $\theta < 1$ there exist $\lambda > 0$ such that for all $t > 0$
$$
\|f(t) - \M(f_0)\, G \|_{L^2(m_{0}^\theta)} \lesssim \exp(-\lambda t^\sigma) \,\|f_0 - \M(f_0)\, G \|_{L^2(m_{0})}.
$$
\end{theo}

\begin{proof}[Proof of Theorem~\ref{th:SLsmall}.]  
We write the representation formulas (taken from \cite{MS*}) 
\begin{align}
& S_\LL(t)f =  \Pi  f + \sum_{\ell=0}^{5} (I-\Pi)  S_\BB *  (\AA S_\BB)^{(*\ell)} (t)f + \TT(t)f, \label{eq:SL=2} \\
&\mbox{where }\, \TT(t) := \lim_{M \to \infty}  {{\rm i} \over  2\pi}
 \int_{a - {\rm i}M}^{a + {\rm i}M}\!\!\!\!  {\rm e}^{zt} \,  
(I-\Pi) \RR_\LL(z) \, (\AA \RR_\BB(z))^{6}  \, dz, \nonumber
\end{align}
for any $f \in D(\LL)$,  $t \ge 0$ and $a > 0$.  

\smallskip
Thanks to Lemma~\ref{lem:SBLpLq}, we know that  each term $\|(I-\Pi)  S_\BB *  (\AA S_\BB)^{(*\ell)} (t)f\|$ in the above expression of $S_{\LL}(t)$ is bounded by 
$\OO(\Theta_{m_0}(t))$. 
In order to conclude, we have to estimate the last term. 

\smallskip
We  introduce the shorthands  $\Phi_1 := R_\LL(I-\Pi)$, $\Phi_{\ell} = \AA R_\BB$, for $2 \le \ell \le 7$, and 
we perform $n$ integration by part in the formula giving $\TT(t)$ to get 
\begin{equation}
  \TT(t) =   {{\rm i} \over  2\pi}
  {1 \over t^n}   \int_{a - {\rm i}\infty}^{a + {\rm i}\infty}  {\rm e}^{zt} \,  
 {d^n \over dz^n} \bigl(  \prod_{i=1}^7 \Phi_i(z)   \bigr)   \, dz.
\end{equation}
Using the fact that all the functions appearing in the integral  are bounded on the imaginary axis, together with the resolvent identity
$$
R_\Lambda^{n} (z) := {d^n \over dz^n}  R_\Lambda(z) = n!  R_\Lambda(z)^n, 
$$
we find in $\BBB(X_0)$, thanks to Leibniz formula and for any $z = x + {\rm i}y \in {\Bbb C}$ with $0 \leq y \leq 1$,
\begin{align*}
\|{d^n \over dz^n} \bigl(  \prod_{\ell=1}^7 \Phi_\ell(z)\bigr)  \| 
& \le 7^n \sup_{\alpha \in \N^{n}, |\alpha| =  n}   \|  \prod_{\ell=1}^7  {d^{\alpha_\ell} \over dz^{\alpha_\ell}} \Phi_\ell (z) \| 
\\
&\hskip-2cm \le
7^n \, \sup_{\alpha \in \N^{7}, |\alpha| =  n}   \prod_{i=1}^7   \| \alpha_1 ! (I-\Pi) \, R_\LL^{1+\alpha_1}  \alpha_2 ! \AA R_\BB^{1+\alpha_2}  \, ... \, \alpha_7! \AA R_\BB^{1+\alpha_7}  (z)\| 
\\
&\hskip-2cm\le C^n \, (n!)^{1+j}  \, \langle y \rangle^{-3/2}\, ,
\end{align*}
where in the last step we have used Lemma \ref{lem:RL} for some integer $j$ which will be fixed later.
Next,  using the bound $n! \le (C \, n)^n$, we get 
$$
\|\TT(t) \|  \le  C^n \, n^{(1+j)n}  \, t^{-n} \qquad \forall \, t > 0, \,\, \forall \, k \ge 0.
$$
For any $t \geq t^*$, where $t^*$ is large enough and depends on $j$, we choose $n \in \N$ such that 
$$
t \ge 2 C n^{1+j} \ge t/2, 
$$
and we obtain 
$$
\|\TT(t) \|\le (Cn^{1+j} t^{-1})^n \le (1/2)^n  \le  (1/2)^{(t/4C)^{1 \over j+1}} \qquad \forall \, t > 0.
$$
As a consequence, with the choice $j := \lfloor[2(1-\gamma)/\gamma\rfloor + 1$,  we have proved that for all $t \geq 0$ we have
$$
\|\TT(t) \|  \le   \exp(-\lambda \, t^{1 \over 1+j}), 
$$
which clearly ends the proof. 
\end{proof}

%
%
%
%
%
%
%
%
%

\bigskip 
\bibliographystyle{acm}

\end{document}